\documentclass[11pt,reqno]{amsart}

\usepackage{enumerate}
\usepackage{amsmath, amssymb, amsthm}
\usepackage{mathrsfs}
\usepackage{esint}
\usepackage{mathtools}
\usepackage{hyperref}
\hypersetup{hidelinks}
\usepackage{bm}
\usepackage{xcolor}
\definecolor{unresolvedred}{RGB}{190,0,0}

\newtheorem{thm}{Theorem}[section]
\newtheorem{corollary}[thm]{Corollary}
\newtheorem{lem}[thm]{Lemma}

\theoremstyle{definition}
\newtheorem{defn}[thm]{Definition}
\newtheorem{example}[thm]{Example}
\newtheorem{assumption}[thm]{Assumption}
\theoremstyle{remark}
\newtheorem{rem}[thm]{Remark}

\newcommand\bN{\mathbb{N}}

\newcommand\bR{\mathbb{R}}

\newcommand\cA{\mathcal{A}}

\newcommand\cF{\mathcal{F}}

\newcommand\cU{\mathcal{U}}

\DeclareMathOperator*{\esssup}{ess\,sup}
\DeclareMathOperator*{\essinf}{ess\,inf}

\numberwithin{equation}{section}
\newcommand{\mysection}[1]{\section{#1}}

\begin{document}

\title[A weighted $\mathrm{L}_2$-theory for Cauchy problems with non-integrable P-D operators]
{Well-Posedness for Cauchy Problems with Singular Time-Measurable Pseudo-Differential Operators in Quasi-decreasing Weighted $\mathrm{L}_2$-Spaces}

\author[J.-H. Choi]{Jae-Hwan Choi}
\address[J.-H. Choi]{School of Mathematics, Korea Institute for Advanced Study, 85 Hoegiro, Dongdaemun-gu, Seoul 02455, Republic of Korea}
\email{jhchoi@kias.re.kr}

\author[I. Kim]{Ildoo Kim}
\address[I. Kim]{Department of Mathematics, Korea University, 145 Anam-ro, Seongbuk-gu, Seoul, 02841, Republic of Korea}
\email{waldoo@korea.ac.kr}

\thanks{J.-H. Choi has been supported by a KIAS Individual Grant (MG102701) at Korea Institute for Advanced Study.
I. Kim has been supported by the National Research Foundation of Korea (NRF) grant funded by the Korean government (MSIT) (RS-2025-16065358).}



\begin{abstract}
This study examines Cauchy problems governed by highly singular, time-measurable pseudo-differential operators (singular measurable families of Fourier multipliers). 
We show that the symbols of these operators can exhibit arbitrary blow-up behavior. 
In particular, we prove the existence and uniqueness of solutions even when the symbols grow super-exponentially in time and frequency. 
As a concrete application, we solve evolutionary equations driven by fractional Laplacians of any negative order. 
Additionally, we establish unique strong solutions under the sole condition that the symbol is locally integrable in frequency, even in the presence of severe blow-up at the initial time.
\end{abstract}
\maketitle

\mysection{Introduction}

Pseudo-differential operators constitute an indispensable framework in the modern theory of partial differential equations (PDEs), emerging naturally as a profound generalization of classical differential operators when analyzing equations with variable coefficients. 
While classical differential operators are strictly local, pseudo-differential operators are fundamentally characterized by their global, non-local nature. 
This inherent non-locality means that standard localization techniques and classical perturbation methods—which heavily rely on freezing coefficients or analyzing immediate neighborhoods—are rendered ineffective. 
Consequently, the presence of these operators not only introduces significant analytical complexities but also demands the development of far more sophisticated mathematical machinery to overcome the limitations of classical PDE theories (\textit{cf.} \cite{E1,H1,K1,T1}).

The significance of this work is deeply rooted in its connections to probability theory. 
Specifically, the connection between pseudo-differential operators and the generators of stochastic processes is most clearly seen in the study of Feller processes and Lévy processes. 
By Courrège's theorem, the infinitesimal generator of a Feller process—provided its domain contains the space of smooth functions with compact support, $C_c^\infty(\mathbb{R}^n)$—can be uniquely represented as a pseudo-differential operator. 
The symbol of this operator, typically a continuous negative definite function, corresponds directly to the characteristic exponent of the associated stochastic process via the Lévy-Khintchine formula. 
This linkage is of fundamental importance because it allows the powerful, rigorous machinery of harmonic analysis, symbol calculus, and operator theory to be seamlessly applied to probability theory. 
Consequently, researchers can analyze the regularity of transition densities, the pathwise behavior of jump processes, and the existence and uniqueness of solutions to complex partial differential equations simply by studying the analytic and algebraic properties of their corresponding pseudo-differential symbols.
For instance, see \cite{A1,BGT1,J1,Kol1} for details.

Historically, a substantial body of literature has been dedicated to relaxing the regularity conditions imposed on the symbols of pseudo-differential operators. 
This drive stems from the inherent need to analyze partial differential equations modeling highly irregular or singular phenomena. In the standard framework, establishing well-posedness for general data within $\mathrm{L}_p$-spaces heavily relies on the continuity—or at least the closedness—of these operators. 
This topological requirement has been widely considered indispensable, as classical existence and uniqueness proofs fundamentally depend on density arguments, where approximations by smooth functions are passed to the limit. 
Indeed, if one strictly seeks classical strong solutions, operator continuity remains virtually an indispensable prerequisite.
However, the breakdown of standard operator boundedness does not preclude the existence of meaningful solutions. 
Modern PDE theory recognizes a diverse spectrum of solvability regimes beyond the rigid confines of strong solutions. 
More importantly, even in the complete absence of operator continuity, one can still secure both the existence and the uniqueness of solutions by carefully constructing and identifying appropriate, specialized function spaces that naturally accommodate the given operators.
Motivated by this perspective, we introduce a novel concept of solvability that fundamentally departs from conventional frameworks. 
We propose a new class of solutions which is strictly weaker than standard distributional weak solutions defined via pairing with test functions. 
Instead, our framework characterizes the solution through a rigorous approximation procedure directly within the frequency domain (see Definition \ref{spectral-limit soluion}), thereby successfully bypassing the traditional continuity bottlenecks.

To rigorously formulate our mathematical framework, let $\psi(t,\xi)$ be a  jointly Lebesgue-measurable complex-valued function on $(0,T) \times \mathbb{R}^d$, and let $w$ be a positive function on $(0,T)$.
In the present work, we establish the existence and uniqueness of solutions in $\mathrm{L}_2\left((0,T) \times \mathbb{R}^d, \frac{w(t)}{t^2}\mathrm{d}t \mathrm{dx}\right)$ to the following Cauchy problem, together with maximal regularity estimates for those solutions:
\begin{align}
								\label{ab eqn}
\begin{cases}
\partial_tu(t,x)=\psi(t,-\mathrm{i}\nabla)u(t,x)+f(t,x),\quad &(t,x)\in(0,T)\times\mathbb{R}^d,\\
u(0,x)=0,\quad & x\in\mathbb{R}^d,
\end{cases}
\end{align}
Here, $\psi(t,-\mathrm{i}\nabla)$ denotes a time-measurable pseudo-differential operator (equivalently, a measurable family of Fourier multipliers). It is governed by the complex-valued symbol $\psi(t,\xi)$ and operates via the Fourier transform as follows:
\begin{align}
								\label{ab op}
\psi(t,-\mathrm{i}\nabla)u(t,x):=\cF_\xi^{-1}\left[\psi(t,\xi)\cF_y[u(t,y)](\xi)\right](x).
\end{align}
Crucially, our framework relies solely on a degenerate ellipticity condition for the symbol. 
We deliberately impose no regularity or boundedness constraints on $\psi(t,\xi)$ with respect to either the time variable $t$ or the frequency variable $\xi$. 
Remarkably, our approach remains entirely valid even without the standard assumption of local integrability for the symbol. 
Furthermore, the weight $w$ under consideration strictly falls outside the classical Muckenhoupt $A_2(\bR^d)$ class.
At first glance, one might assume that the analysis is simplified by restricting the problem to an $\mathrm{L}_2$-setting. 
Indeed, classical multiplier theory dictates that a Fourier multiplier is bounded on $\mathrm{L}_2$ if and only if its corresponding symbol is uniformly bounded. 
However, as previously emphasized, our framework completely abandons the requirement of operator continuity. Our operators are not even densely defined in $\mathrm{L}_2$.
Consequently, standard arguments based on symbol boundedness break down, making the problem fundamentally non-trivial despite the $\mathrm{L}_2$-framework.
Moreover, the intricate singularities of the symbols preclude the application of many elegant functional analytic tools typically available in Hilbert spaces. 
Instead, we must carry out all computations and arguments using elementary methods, relying heavily on Plancherel's theorem, Fubini's theorem, the generalized Minkowski inequality, and the Cauchy--Bunyakovsky--Schwarz inequality. 
Ultimately, while the mathematical machinery we employ is strictly elementary, it yields remarkably powerful and far-reaching results.

Our research group has extensively investigated this class of evolution equations in recent years (see, for instance, \cite{CJH KID 2023,CJH LJB KID 2023,CJH 2024,KID SBL KHK 2015, KID SBL KHK 2016,KID KHK 2016,KID 2018,KKK0,KKK1}). 
To highlight the active nature of this topic within mathematical analysis, we also direct the reader to selected contributions from other researchers \cite{DHP1,AB1,MP1,Z1,DK1,NVW1,BNVW1}. 
Notably, these prior investigations were primarily developed within an $L_p$-framework. 
Consequently, relying on principles akin to the Mikhlin multiplier theorem, they inherently necessitate strong regularity assumptions on the symbols. 
Moreover, treating these operators via an integral approach strictly demands some level of local integrability. 
To overcome these limitations, the present work specifically targets the $p=2$ case. Focusing on this special exponent allows us to achieve our primary objective: completely dispensing with both the regularity and integrability requirements on the symbols.

While standard differential operators (such as the Laplacian) exhibit smooth polynomial growth exclusively as $\vert{}\xi\vert{} \to \infty$—yielding stationary domains homeomorphic to classical Sobolev spaces—our symbols $\psi(t,\xi)$ are allowed to undergo localized spatio-temporal blow-ups at finite times or frequencies. 
From an operator-theoretic perspective, such localized singularities trigger a severe topological collapse of the operator domains $D(\psi(t,-i\nabla))$ as $t$ approaches singular times, thereby destroying the uniform sectoriality, resolvent bounds, and domain stability required by Kato's evolution family framework. 
Consequently, these operators cannot be continuously mapped within any classical Sobolev scale $H^s$, rendering traditional variational and semigroup methodologies fundamentally inapplicable.
This severe breakdown necessitates a novel theoretical structure capable of handling such severe singularities.

To illustrate the efficacy of our proposed framework, even without assuming local integrability of the symbols, consider the following concrete toy model.
Consider the highly singular symbol
$$
\psi(t,\xi)=-t^\alpha \exp\left( t^\beta \right) |\xi|^{\gamma} \exp\left( |\xi|^{\delta} \right) 
$$
and the corresponding Cauchy problem:
\begin{align}
									\label{ab eqn 2}
\begin{cases}
\partial_tu(t,x)= -t^\alpha \exp\left(  t^\beta \right) (-\Delta)^{\gamma/2} \exp\left( (-\Delta)^{\delta/2} \right)  u(t,x)+f(t,x),\quad &(t,x)\in(0,T)\times\mathbb{R}^d,\\
u(0,x)=0,\quad & x\in\mathbb{R}^d.
\end{cases}
\end{align}
Remarkably, our results guarantee that this equation can be uniquely solved in the aforementioned approximating sense for any choice of real parameters $\alpha,\beta,\gamma, \delta \in \mathbb{R}$ (see Theorem \ref{main appl 1}). 
Furthermore, we can even establish the existence of a unique strong solution to \eqref{ab eqn 2} under the remarkably broad conditions $\alpha \in (-1,0]$, $\beta=0$ $\gamma \in (-d/2,\infty)$, and $\delta \in [0,\infty)$ (Theorem \ref{main appl 3}).

Although the toy model we presented features irregularities at just one or two specific points to demonstrate blow-up behavior, our broader framework is not bound by these limitations. 
In fact, a key advantage of our methodology is its capacity to handle symbols possessing arbitrary levels of irregularity and blow-up. 
Having laid out this motivating context, we now conclude the motivating discussion and proceed to the rigorous theory.

We emphasize that the restriction to homogeneous initial data $u(0,x)=0$ in \eqref{ab eqn} is not merely a technical simplification, but a structural necessity dictated by the extreme initial singularities of the symbols. When $\psi(t,\xi)$ fails to be integrable near $t=0$, the heuristic propagator $\exp\left(\int_0^t \psi(r,\xi)\,dr\right)$ becomes ill-defined even in the sense of tempered distributions. In such hyper-singular regimes, standard trace operators at $t=0$ lose their classical meaning, rendering non-zero initial data ill-posed without imposing additional temporal integrability constraints on the symbol. For more details, see Remark \ref{zero initial reason} below.

The remainder of this paper is structured as follows: Section 2 introduces the function spaces and three distinct notions of solutions. Section 3 contains our main well-posedness theorems, while Section 4 illustrates these results through concrete fractional toy models. 
Sections 5 through 7 develop the necessary a priori estimates, uniqueness criteria, and dense subclasses. 
Finally, Section \ref{pf main thm} provides the proofs of the main theorems; for logical cohesion, Theorem \ref{main thm 2} is established first, followed sequentially by Theorems \ref{main thm} and \ref{main thm 3}.

\vspace{2mm}
We close this section by summarizing the standard notation that will be used throughout the remainder of the paper.

\begin{itemize}
\item We denote the sets of natural numbers and integers by $\mathbb{N}$ and $\mathbb{Z}$, respectively. 
The $d$-dimensional Euclidean space is represented by $\mathbb{R}^d$, with elements written as $x = (x^1, \ldots, x^d)$. 
For a sufficiently smooth function $u(x)$ and multi-indices $\alpha = (\alpha_1, \ldots, \alpha_d)$ with non-negative integer components, we denote the partial derivatives by $u_{x^i} = \frac{\partial u}{\partial x^i} = D_i u$ and $D^\alpha u = D_1^{\alpha_1} \cdots D_d^{\alpha_d} u$. 
The spatial gradient is written as $u_x = (u_{x^1}, \ldots, u_{x^d})$.

\item For a given domain $\mathcal{O} \subset \mathbb{R}^d$ and a normed space $F$, $C(\mathcal{O}; F)$ is the space of all $F$-valued continuous functions on $\mathcal{O}$. 
This space is equipped with the supremum norm $\|u\|_{C} := \sup_{x \in \mathcal{O}} \|u(x)\|_F < \infty$.

\item Given a measure space $(X, \mathcal{M}, \mu)$, a normed space $F$, and $p \in [1, \infty)$, $\mathrm{L}_p(X, \mathcal{M}, \mu; F)$ denotes the space of all $F$-valued, $\mathcal{M}^\mu$-measurable functions $u$ satisfying
$$
\|u\|_{\mathrm{L}_p(X, \mathcal{M}, \mu; F)} := \left( \int_X \|u(x)\|_F^p \, \mu(\mathrm{d}x) \right)^{1/p} < \infty.
$$
Here, $\mathcal{M}^\mu$ represents the $\mu$-completion of the $\sigma$-algebra $\mathcal{M}$. 
When $p = \infty$, the space consists of functions with a finite essential supremum:
$$
\|u\|_{L_\infty(X, \mathcal{M}, \mu; F)} := \inf \left\{ \nu \geq 0 : \mu(\{ x : \|u(x)\|_F > \nu \}) = 0 \right\} < \infty.
$$
We often omit the measure space and $\sigma$-algebra from the notation when they are evident from context. 
By default, we work with the completion of the Borel $\sigma$-algebra associated with the underlying topology.

\item In $\mathbb{R}^d$, we primarily rely on Lebesgue measurable sets (the completion of the Borel $\sigma$-algebra under the standard Euclidean metric). The Lebesgue measure of a measurable set $\mathcal{O} \subset \mathbb{R}^d$ is written as $|\mathcal{O}|$. Within the context of Euclidean spaces, we use ``measurable'' to implicitly mean ``Lebesgue-measurable.''

\item Suppose $\mathcal{O} \subset \mathbb{R}^d$ and $\mathcal{N} \subset \mathcal{O}$, where $\mathcal{N}$ is a null set (i.e., contained within a Borel set of Lebesgue measure zero). A function $f$ defined on $\mathcal{O} \setminus \mathcal{N}$ is considered measurable and defined almost everywhere on $\mathcal{O}$ if it coincides with some globally measurable function $g$ on $\mathcal{O} \setminus \mathcal{N}$.

\item We use the notation $\alpha \lesssim \beta$ to indicate that $\alpha \leq N \beta$ for some generic constant $N > 0$. 
The exact value of $N$ may change from line to line. 
Subscripts, as in $\alpha \lesssim_{a,b,c} \beta$, mean that the implicit constant $N$ depends solely on the parameters $a$, $b$, and $c$.
The precise dependencies of these constants are always explicitly stated in our theorems, lemmas, and corollaries.

\item For a measurable function $f$ on $\mathbb{R}^d$, its $d$-dimensional Fourier transform and inverse Fourier transform are defined as
$$
\mathcal{F}[f](\xi) := \frac{1}{(2\pi)^{d/2}}\int_{\mathbb{R}^{d}} \mathrm{e}^{-i\xi \cdot x} f(x) \mathrm{d}x \quad \text{and} \quad \mathcal{F}^{-1}[f](x) := \frac{1}{(2\pi)^{d/2}}\int_{\mathbb{R}^{d}} \mathrm{e}^{ ix \cdot \xi} f(\xi) \mathrm{d}\xi.
$$
For a space-time function $f(t,x)$ defined on $(0,T)\times \mathbb{R}^d$, we denote the spatial Fourier transform by
$$
\mathcal{F}_x[f](\xi) := \mathcal{F}[f(t,\cdot)](\xi) = \frac{1}{(2\pi)^{d/2}}\int_{\mathbb{R}^{d}} \mathrm{e}^{ -i\xi \cdot x} f(t,x) \mathrm{d}x,
$$
and the spatial inverse Fourier transform by
$$
\mathcal{F}_\xi^{-1}[f](x) := \mathcal{F}^{-1}[f(t,\cdot)](x) = \frac{1}{(2\pi)^{d/2}}\int_{\mathbb{R}^{d}} e^{ ix \cdot \xi} f(t,\xi) \mathrm{d}\xi.
$$
By the Riesz-Thorin and Plancherel theorems, these transforms extend naturally to $L_1(\mathbb{R}^d) + \mathrm{L}_2(\mathbb{R}^d)$, and we retain the symbols $\mathcal{F}$ and $\mathcal{F}^{-1}$ for these extensions.

\item We let $\mathrm{C}_c^{\infty}(\mathbb{R}^d)$ denote the space of compactly supported smooth functions, $\mathcal{S}(\mathbb{R}^d)$ the Schwartz space, and $\mathcal{S}'(\mathbb{R}^d)$ the space of tempered distributions. The set of Schwartz functions whose Fourier transforms are compactly supported is denoted by $\mathcal{F}^{-1}\mathrm{C}_c^\infty(\mathbb{R}^d)$; explicitly,
$$
f \in \mathcal{F}^{-1}\mathrm{C}_c^\infty(\mathbb{R}^d) \iff f \in \mathcal{S} \text{ and } \mathcal{F}[f] \in \mathrm{C}_c^\infty(\mathbb{R}^d).
$$

\item For real numbers $a,b \in \mathbb{R}$, we use the shorthand
$$
a \wedge b := \min\{a,b\} \quad \text{and} \quad a \vee b := \max\{a,b\}.
$$
For a complex number $z \in \mathbb{C}$, its real part, imaginary part, and complex conjugate are denoted by $\Re[z]$, $\Im[z]$, and $\bar{z}$, respectively.

\item For a complex-valued function $u$ defined on a subset $E$ of Euclidean space, we write $u \equiv 1$ to indicate that $u(t) = 1$ for all $t \in E$.

\item For topological vector spaces $X$ and $Y$, the notation $X \hookrightarrow Y$ means that $X$ is continuously embedded in $Y$.
\end{itemize}

\mysection{Setting and solutions}
											\label{main section}

In this paper, $T \in (0,\infty]$ denotes the terminal time and $d \in \mathbb{N}$ represents the spatial dimension. 
Because $T$ can be infinite, our evolution equations are formulated on the open time interval $(0,T)$. Unless otherwise specified, all functions are assumed to be complex-valued. 
The sole exception concerns weight functions, which are strictly real-valued and positive almost everywhere on $(0,T)$.
For a complex-valued symbol $\psi(t,\xi)$ defined on $(0,T) \times \mathbb{R}^d$, we define the associated time-measurable pseudo-differential operator (or measurable family of Fourier multipliers), $\psi(t,-\mathrm{i}\nabla)$, by
\begin{align}
									\label{20230522 01}
\psi(t,-\mathrm{i}\nabla)u(t,x):=\cF_\xi^{-1}\left[\psi(t,\xi)\cF_y[u(t,y)](\xi)\right](x).
\end{align}
While interpolation allows the Fourier transform and its inverse to be defined in a strong sense on $\mathrm{L}_p(\mathbb{R}^d)$ for any $p \in [1,2]$, we strictly confine our framework to $L_1(\mathbb{R}^d)$ or $L_2(\mathbb{R}^d)$. 
It should be emphasized that distribution-valued functions are entirely absent from this work. 
Consequently, whenever the Fourier transform of a function $f$ is taken, both $f$ and its transform are understood as standard complex-valued functions, with the underlying assumption that $f \in L_1(\mathbb{R}^d) + L_2(\mathbb{R}^d)$.
Furthermore, although the symbols $\psi$ handled in our main results are typically unbounded, the associated operators can still be interpreted in a strong sense provided that $\mathcal{F}[u]$ exhibits sufficient decay.

Before presenting our main theorem, we review the notation for mixed-norm spaces. 
Let $p \in [1,\infty]$ and $q \in [1,\infty)$ be given exponents, and let $W$ be a non-negative weight function defined almost everywhere on $(0,T) \times \mathbb{R}^d$.
To avoid cumbersome notation, we write
$$
\mathrm{L}_{p,q}\left( (0,T) \times \bR^d,   W(t,x) \mathrm{d}t   \mathrm{d}x \right)
$$ 
as a convenient shorthand for the space $\mathrm{L}_{p}\left( (0,T), \mathrm{d}t ; L_q \left(\mathbb{R}^d, W(t,x)\mathrm{d}x \right)\right)$. 
We require functions in this space to be jointly measurable because iterated measurability alone does not guarantee joint measurability.
Explicitly, this space comprises all jointly measurable functions $u$ defined a.e. on $(0,T) \times \mathbb{R}^d$ for which the mixed norm is finite:
\[
\|u\|_{\mathrm{L}_{p,q}\left( (0,T) \times \mathbb{R}^d, W(t,x)\mathrm{d}t \mathrm{d}x \right)} := \left\| \left(\int_{\mathbb{R}^d} |u(\cdot,x)|^q W(\cdot,x) \mathrm{d}x \right)^{1/q} \right\|_{\mathrm{L}_p(0,T)} < \infty.
\]
Written out in full detail, $u \in \mathrm{L}_{p,q}\left( (0,T) \times \mathbb{R}^d, W(t,x)\mathrm{d}t \mathrm{d}x \right)$ if and only if
$$
\|u\|_{ \mathrm{L}_{p,q} } := 
\begin{cases}
\left(\int_{0}^{T} \left(\int_{\mathbb{R}^d} |u(t,x)|^q W(t,x) \mathrm{d}x \right)^{p/q} \mathrm{d}t \right)^{1/p} < \infty, & \text{for } p \in [1,\infty), \\
\mathop{\mathrm{ess\,sup}}_{t \in (0,T)} \left(\int_{\mathbb{R}^d} |u(t,x)|^q W(t,x) \mathrm{d}x \right)^{1/q} < \infty, & \text{for } p = \infty.
\end{cases}
$$

\begin{rem}
In our analysis, we employ weighted $\mathrm{L}_p$-spaces equipped with the weight function $|\psi(t,\xi)|$. 
Because we allow for highly general symbols, this weight may vanish on a subset of positive Lebesgue measure. 
In standard measure theory, this degeneracy introduces a critical issue: equivalence classes formed with respect to the weighted measure allow functions to differ arbitrarily on the zero-weight set. 
This ambiguity fundamentally conflicts with the standard notion of uniqueness for PDE solutions.
To circumvent this problem, we treat these weighted spaces as topological vector spaces endowed with seminorms rather than true norms. 
We strictly define our equivalence classes based on almost-everywhere equality with respect to the standard Lebesgue measure, independent of where the weight vanishes. 
\end{rem}

\begin{defn}[Weighted local $\mathrm{L}_{p,2}$-spaces]
											\label{weight local space}
Let $p \in [1,\infty]$.
\begin{itemize}
\item We define the local space $\mathrm{L}_{p,2,loc}\left( (0,T) \times \mathbb{R}^d, W(t,x) \mathrm{d}t \mathrm{d}x \right)$ as the set of functions $u$ satisfying
$$
u \in \mathrm{L}_{p,2}\left( (0,T') \times \mathbb{R}^d, W(t,x) \mathrm{d}t \mathrm{d}x \right) \quad \text{for all } T' \in (0,T).
$$ 
Naturally, this space forms a topological vector space equipped with a family of seminorms indexed by $T'$. 
In particular, we say a sequence $\{u_n\}$ converges to $u$ in $\mathrm{L}_{p,2,loc}\left((0,T) \times \mathbb{R}^d, W(t,x)\mathrm{d}t\mathrm{d}x \right)$ if, for every $T' \in (0,T)$, it converges with respect to the corresponding seminorm. 
For $p < \infty$, this means:
$$
\lim_{n \to \infty}\int_0^{T'} \left(\int_{\mathbb{R}^d}|u(t,x)-u_n(t,x)|^2 W(t,x)\mathrm{d}x \right)^{p/2}\mathrm{d}t=0.
$$
Throughout this paper, we predominantly focus on the case where $p=2$, adopting the abbreviated notation $\mathrm{L}_{2,loc}\left( (0,T) \times \mathbb{R}^d, W(t,x) \mathrm{d}t \mathrm{d}x \right)$ in place of $\mathrm{L}_{2,2,loc}\left( (0,T) \times \mathbb{R}^d, W(t,x) \mathrm{d}t \mathrm{d}x \right)$.

\item Next, we introduce a local function class defined by the properties of its spatial Fourier transform. We write 
$$
u \in \mathcal{F}_x^{-1}\mathrm{L}_{p,2,loc}\left( (0,T) \times \mathbb{R}^d,  W(t,\xi)\mathrm{d}t \mathrm{d}\xi \right)
$$ 
if and only if, for every $T' \in (0,T)$, one has $u(t,\cdot)\in L_2(\mathbb{R}^d)$ for almost every $t\in(0,T')$, and the weighted mixed-norm of its Fourier transform is finite. Explicitly, this second condition requires:
$$
\|\mathcal{F}_x[u]\|_{ \mathrm{L}_{p,2}\left( (0,T') \times \mathbb{R}^d, W(t,\xi) \mathrm{d}t \mathrm{d}\xi \right)} := 
\begin{cases}
\left(\int_{0}^{T'} \left(\int_{\mathbb{R}^d} |\mathcal{F}_x[u(t,x)](\xi)|^2 W(t,\xi) \mathrm{d}\xi \right)^{p/2} \mathrm{d}t \right)^{1/p} < \infty, & \text{for } p \in [1,\infty), \\
\mathop{\mathrm{ess\,sup}}_{t \in (0,T')} \left(\int_{\mathbb{R}^d} |\mathcal{F}_x[u(t,x)](\xi)|^2 W(t,\xi) \mathrm{d}\xi \right)^{1/2} < \infty, & \text{for } p = \infty.
\end{cases}
$$
\end{itemize}
Following standard conventions, we omit the measure notation $W(t,\xi) \mathrm{d}t \mathrm{d}\xi$ whenever the weight is trivial (i.e., $W(t,\xi) \equiv 1$).

\end{defn}

\begin{rem}
Before proceeding, we briefly reiterate our convention concerning the Fourier transform. 
To bypass the technical difficulties associated with tempered distribution-valued functions, we strictly regard all functions and their Fourier transforms as standard complex-valued functions. 
Consequently, we restrict the integrability exponent for the frequency variable to $2$, which is reflected in our use of the space $\mathcal{F}^{-1}\mathrm{L}_{p,2,loc}\left( (0,T) \times \mathbb{R}^d, W(t,\xi)\mathrm{d}t \mathrm{d}\xi \right)$.
While it is possible to extend these definitions to a general exponent $q \in [1,\infty]$ by employing classes of linear functionals broader than tempered distributions (see, e.g., \cite{CK 2024, CK 2024-2}), such generality is unnecessary for the scope of this paper. 
We focus exclusively on the $q=2$ case because Plancherel's theorem ensures that the Fourier transform is defined without any ambiguity. 
In particular, when the weight $W$ depends solely on the time variable $t$, Plancherel's theorem yields the direct identity:
$$
\mathcal{F}_x^{-1}\mathrm{L}_{2,loc}\left( (0,T) \times \mathbb{R}^d, W(t)\mathrm{d}t \mathrm{d}x \right) 
=\mathrm{L}_{2,2,loc}\left( (0,T) \times \mathbb{R}^d, W(t)\mathrm{d}t \mathrm{d}x \right)
=\mathrm{L}_{2,loc}\left( (0,T) \times \mathbb{R}^d, W(t)\mathrm{d}t \mathrm{d}x \right).
$$
\end{rem}

\begin{defn}[Strong Solution]
							\label{strong solution}
Let $u \in \mathrm{L}_{2,loc}\left((0,T) \times \bR^d  \right)$.
We say that $u$ is a \emph{strong solution} to equation \eqref{ab eqn} if
\begin{align}
										\label{20230213 30}
u(t,x) = \int_0^t \left(\psi(s,-\mathrm{i}\nabla)u(s,x)+f(s,x)   \right) \mathrm{d}s \quad \text{a.e.}~ (t,x) \in (0,T) \times \bR^d.
\end{align}
\end{defn}

\begin{rem}
Alternatively, one could formulate a strong solution in the frequency domain as follows:
$$
\mathcal{F}[u(t,\cdot)](\xi) = \int_0^t \left(\psi(s,\xi)\mathcal{F}[u(s,\cdot)](\xi)+\mathcal{F}[f(s,\cdot)](\xi) \right) \mathrm{d}s \quad \text{a.e. } (t,\xi) \in (0,T) \times \mathbb{R}^d.
$$
If we formally interchange the time integral and the Fourier transform, Plancherel's theorem allows us to treat the two formulations in essentially the same way. 
\end{rem}

\begin{rem}
The space $\mathrm{L}_{2,loc}((0,T) \times \mathbb{R}^d)$ is not the optimal setting for defining strong solutions, primarily because the integral term $\int_0^t \psi(s,-\mathrm{i}\nabla)u(s,x) \mathrm{d}s$ is not necessarily well-defined for an arbitrary $u \in \mathrm{L}_{2,loc}((0,T) \times \mathbb{R}^d)$. 
However, we can readily deduce the necessary properties of strong solutions from Definition \ref{strong solution}. 
If $u$ is a strong solution, it follows from \eqref{20230522 01} and \eqref{20230213 30} that for almost every $t \in (0,T)$, the mapping
$$
\xi \mapsto \psi(t,\xi)\mathcal{F}[u(t,\cdot)](\xi)
$$
belongs to $L_1(\mathbb{R}^d) + L_2(\mathbb{R}^d)$. 
Furthermore, we require that
$$
\psi(\cdot,-\mathrm{i}\nabla)u(\cdot,x) \in \mathrm{L}_{1,loc}\left((0,T)\right) \quad \text{for almost every } x \in \mathbb{R}^d.
$$
Because these conditions are generally not met by all functions within $\mathrm{L}_{2,loc}((0,T) \times \mathbb{R}^d)$, we must identify a more suitable subspace to properly classify our strong solutions. 
A natural and optimal choice for this is the space $\mathcal{F}_x^{-1}\mathrm{L}_{1,2,loc}\left( (0,T) \times \mathbb{R}^d, |\psi(t,\xi)|^2 \mathrm{d}t \mathrm{d}\xi \right)$, which was introduced in Definition \ref{weight local space} with $W(t,\xi) =|\psi(t,\xi)|^2$. 
\end{rem}

\begin{rem}
									\label{continuous modification}
By the Lebesgue differentiation theorem, 
we can deduce from \eqref{20230213 30} that
$$
\partial_t u(t,x) = \psi(t,-\mathrm{i}\nabla)u(t,x) + f(t,x) 
$$ 
holds for almost every $(t,x) \in (0,T) \times \mathbb{R}^d$.
Furthermore, if we introduce the function
$$\tilde{u}(t,x) = \int_0^t \left( \psi(s,-\mathrm{i}\nabla)u(s,x) + f(s,x) \right) \mathrm{d}s,
$$
it naturally follows that $u(t,x)$ and $\tilde{u}(t,x)$ coincide almost everywhere on $(0,T) \times \mathbb{R}^d$. 
It is also apparent that $\tilde{u}(t,x)$ is continuous in $t$ for any given $x$. 
As a result, any strong solution $u$ to \eqref{ab eqn} possesses a modification that is continuous everywhere in time. Therefore, without loss of generality, we can assume that \eqref{20230213 30} holds for all $t \in (0,T)$ for almost every spatial point $x \in \mathbb{R}^d$.

\end{rem}

\begin{defn}[Spatial Fourier Weak Solution]
									\label{space weak solution}
A measurable function $u \in \mathrm{L}_{2,loc}((0,T) \times \mathbb{R}^d)$ is termed a  \emph{spatial Fourier weak solution} to equation \eqref{ab eqn} provided that, for all test functions $\varphi \in \mathcal{F}^{-1}\mathrm{C}_c^\infty(\mathbb{R}^d)$, the following identity holds for almost every $t \in (0,T)$:
\begin{align}
										\label{weak formulation}
\left(u(t,\cdot),\varphi\right)_{\mathrm{L}_2(\bR^d)} =  \int_0^t \left(u(s,\cdot) , \overline \psi(s,-\mathrm{i}\nabla)\varphi \right)_{\mathrm{L}_2(\bR^d)} \mathrm{d}s 
+ \int_0^t \left( f(s,\cdot),\varphi\right)_{\mathrm{L}_2(\bR^d)} \mathrm{d}s
\quad a.e.~t\in (0,T).
\end{align}
Here, $(\cdot,\cdot)_{\mathrm{L}_2(\mathbb{R}^d)}$ represents the standard $\mathrm{L}_2(\mathbb{R}^d)$-inner product, defined as
$$
\left(u(t,\cdot),\varphi\right)_{\mathrm{L}_2(\mathbb{R}^d)} = \int_{\mathbb{R}^d} u(t,x) \overline{\varphi(x)} \mathrm{d}x.
$$ 
Furthermore, the conjugated operator is defined via the Fourier transform as
$$
\overline\psi(s,-\mathrm{i}\nabla)\varphi(x)= \mathcal{F}^{-1}\left[\overline \psi(s,\cdot) \mathcal{F}[\varphi] \right](x), 
$$ 
where $\overline\psi(s,\xi)$ is the complex conjugate of the symbol $\psi(s,\xi)$.
\end{defn}

\begin{rem}
Recall that $\mathcal{F}^{-1}\mathrm{C}_c^\infty(\mathbb{R}^d)$ represents a subspace of the Schwartz class $\mathcal{S}(\mathbb{R}^d)$ consisting of functions whose Fourier transforms are infinitely differentiable with compact support; that is,
$$
f \in \mathcal{F}^{-1}\mathrm{C}_c^\infty(\mathbb{R}^d) \iff f \in \mathcal{S}(\mathbb{R}^d) \text{ and } \mathcal{F}[f] \in \mathrm{C}_c^\infty(\mathbb{R}^d).
$$
We select this space as our class of test functions to relax the upper bound requirements on the symbol. 
Because $\mathcal{F}^{-1}\mathrm{C}_c^\infty(\mathbb{R}^d)$ is dense in $\mathrm{L}_p(\mathbb{R}^d)$ for all $p \in (1,\infty)$, this choice still effectively bridges the gap between weak and strong solutions. 
Furthermore, utilizing this test function class enables us to solve evolution equations lacking ellipticity. 
For a comprehensive discussion on the existence and uniqueness of these solutions to \eqref{ab eqn}, we refer the reader to our previous work \cite{CK 2024, CK 2024-2}. 
Finally, we note a slight shift in terminology: while we previously referred to these as ``Fourier-space weak solutions" in \cite{CK 2024, CK 2024-2}, we now adopt the term ``spatial Fourier weak solutions."

\end{rem}

\begin{rem}
									\label{strong weak rem}
Formally, integrating both sides of \eqref{20230213 30} demonstrates that any strong solution $u$ to \eqref{ab eqn} is inherently a spatial Fourier weak solution as well. 
However, although Plancherel’s theorem intuitively suggests that a strong solution should naturally satisfy the weak formulation, the commutation of the temporal integral $\int_0^t$ and the Fourier transform is not automatic. 
Justifying this interchange requires applying Fubini’s theorem or relying on the $L_2$-continuity of the Fourier transform, which consequently imposes specific integrability constraints on the symbol $\psi(t,\xi)$. 
The rigorous theoretical foundation for this commutation is provided by Hille’s theorem (\textit{cf.} \cite[Corollary V.5.2]{Yosida}), and a detailed analysis of this operation is deferred to Lemma \ref{strong unique lem}.
 Finally, for the integrals in the weak formulation to be well-defined, the symbol $\psi(t,\xi)$ must be locally integrable with respect to the frequency variable $\xi$.

\end{rem}

Finally, we introduce our weakest notion of a solution, which we term the spectral-limit solution. 
This terminology is motivated by approximation methods in spectral analysis.

\begin{defn}[Spectral-limit solution]
									\label{spectral-limit soluion}
A function $u \in \mathrm{L}_{2,loc}((0,T) \times \mathbb{R}^d)$ is defined as an \emph{spectral-limit solution} to equation \eqref{ab eqn} if there exists a sequence $\{u_n\}_{n \in \mathbb{N}} \subset \mathrm{L}_{2,loc}((0,T) \times \mathbb{R}^d)$ converging to $u$ in $\mathrm{L}_{2,loc}((0,T) \times \mathbb{R}^d)$, such that each $u_n$ is a spatial Fourier weak solution to \eqref{ab eqn} associated with the truncated operator $\psi_n(t,-\mathrm{i}\nabla)$. 
That is, for any test function  $\varphi \in \mathcal{F}^{-1}\mathrm{C}_c^\infty(\mathbb{R}^d)$, 
the identity
\begin{align*}
\left(u_n(t,\cdot),\varphi\right)_{\mathrm{L}_2(\bR^d)} =  \int_0^t \left(u_n(s,\cdot) , \overline \psi_n(s,-\mathrm{i}\nabla)\varphi \right)_{\mathrm{L}_2(\bR^d)} \mathrm{d}s 
+ \int_0^t \left( f(s,\cdot),\varphi\right)_{\mathrm{L}_2(\bR^d)} \mathrm{d}s
\quad a.e.~t\in (0,T)
\end{align*}
holds. 
Here, the operator is given by
$$
\psi_{n}(s,-\mathrm{i}\nabla)u_n(s,x) = \mathcal{F}_{\xi}^{-1} \left[  \psi_{n}(s,\xi) \mathcal{F}_y[u_n(s,y)](\xi)\right](x), 
$$
where the truncated symbol $\psi_n(s,\xi)$ is defined by bounding its real and imaginary parts:
$$
\psi_{n}(s,\xi) := \left( (-n) \vee \Re[\psi(s,\xi)] \wedge n \right)  + \mathrm{i}\left( (-n) \vee \Im[\psi(s,\xi)] \wedge n \right).
$$
Such a sequence $\{u_n\}_{n \in \mathbb{N}}$ is referred to as a spectral-limit sequence (for the solution $u$).
\end{defn}

\begin{rem}
									\label{solution inculsion remark}
 In the preceding definition, we chose the spatial Fourier weak solutions $u_n$ corresponding to the truncated symbols $\psi_n$ as the spectral-limit sequence. 
Theorem \ref{main thm} constructs such a sequence and proves its convergence for the explicit solution \eqref{unique solution}. 
Under the additional operator-integrability hypotheses of Lemma \ref{strong solution lem}, every strong solution is also a spatial Fourier weak solution. We do not assert a general implication from an arbitrary spatial Fourier weak solution to a spectral-limit solution without a separate convergence argument for its truncated approximations.
\end{rem}

\begin{rem}
Identifying a solution as the limit of truncated spectral approximations is a well-established method in the study of pseudo-differential operators and Lévy-type processes. 
Historically, this technique stems from the construction of Feller semigroups for generators with unbounded symbols, where the operator is ``tamed" via frequency-domain cutoffs to maintain analytical control. 
This ``truncation and passage to the limit" strategy—sharing the conceptual foundations of Yosida approximations—facilitates the construction of $L_2$-solutions even when the symbol $\psi(s, \xi)$ exhibits singular growth or non-symmetric, complex-valued behavior. 
Typically, by deriving uniform estimates for the truncated sequence $\{u_n\}$, one can employ weak compactness arguments or the completeness of Bessel potential spaces to establish a unique spectral-limit solution. 
For foundational studies on these spectral truncations, we refer to \cite{FJS1, J1, Ke1}. However, it is crucial to note that such weak compactness strategies are inapplicable in our current setting. 
As discussed in Remark \ref{strong weak rem}, our symbols are generally too singular to admit a standard weak formulation; 
specifically, $\psi(t, \xi)$ may exhibit arbitrary blow-up at any frequency $\xi$, thereby precluding a weak formulation based on test functions.

\end{rem}

\begin{defn}[A weighted $\phi$-potential space]
										\label{phi space}
Let $w$ be a positive function defined almost everywhere on $(0,T)$, and let $\phi$ be a complex-valued function defined almost everywhere on $\mathbb{R}^d$. 
We denote by $H_{2,loc}^{\phi}\left( (0,T) \times \mathbb{R}^d, w(t)\mathrm{d}t\mathrm{d}x \right)$ the space of all measurable functions $u(t,x)$ such that for every $T' \in (0,T)$,
$$
\int_0^{T'} \int_{\mathbb{R}^d} \left|u(t,x)\right|^2 w(t) \mathrm{d}x \mathrm{d}t + \int_0^{T'} \int_{\mathbb{R}^d} \left|\phi(-\mathrm{i}\nabla)u(t,x)\right|^2 w(t) \mathrm{d}x \mathrm{d}t < \infty.
$$
Here, the pseudo-differential operator is defined via the Fourier transform as
$$
\phi(-\mathrm{i}\nabla)u(t,x) = \mathcal{F}_\xi^{-1}\left[ \phi(\xi) \mathcal{F}_y[u(t,y)](\xi) \right](x).
$$ 
By Plancherel's theorem, a function $u$ belongs to 
$$
H_{2,loc}^{\phi}\left( (0,T) \times \mathbb{R}^d, w(t)\mathrm{d}t\mathrm{d}x \right)
$$ 
if and only if, for every $T' \in (0,T)$,
$$
\int_0^{T'} \int_{\mathbb{R}^d} \left|\mathcal{F}[u(t,\cdot)](\xi)\right|^2 w(t) \mathrm{d}\xi \mathrm{d}t 
+ \int_0^{T'} \int_{\mathbb{R}^d} \left|\phi(\xi)\mathcal{F}[u(t,\cdot)](\xi)\right|^2 w(t) \mathrm{d}\xi \mathrm{d}t < \infty. 
$$
When the weight is trivial, \textit{i.e.}, $w(t) \equiv 1$, we simplify the notation to $H_{2,loc}^{\phi}\left( (0,T) \times \mathbb{R}^d \right)$.
\end{defn}

\begin{rem}
The space $H_{2,loc}^{\phi}\left( (0,T) \times \mathbb{R}^d, w(t)\mathrm{d}t\mathrm{d}x \right)$ is complete. 
We emphasize that no structural conditions on the symbol $\phi$—such as smoothness, ellipticity, or growth bounds—are required for the completeness of the space. For any fixed 
$T' \in (0,T)$, the spatial Fourier transform maps $H_{2}^{\phi}\left( (0,T') \times \mathbb{R}^d, w(t)\mathrm{d}t\mathrm{d}x \right)$ isometrically onto the weighted $L_2$ space
$$
L_2\left( (0,T') \times \mathbb{R}^d, \, \left(1 + \vert{}\phi(\xi)\vert{}^2\right)w(t)\mathrm{d}\xi\mathrm{d}t \right)
$$
by Plancherel's identity. Consequently, the space inherits Hilbert (and thus Banach) space completeness directly from the standard completeness of weighted $L_2$ spaces, which demands only the measurability of $\phi$. 
Furthermore, equipped with the countable family of seminorms induced by $T' \in (0,T)$, the local space $H_{2,\mathrm{loc}}^{\phi}\left( (0,T) \times \mathbb{R}^d, w(t)\mathrm{d}t\mathrm{d}x \right)$ forms a Fréchet space.
\end{rem}

\begin{rem}
While Definition \ref{phi space} introduces the $\phi$-potential space for an arbitrary complex-valued function $\phi$ defined almost everywhere on $\mathbb{R}^d$, our primary theorems focus exclusively on non-negative symbols $\phi$ to simplify the presentation. 
This restriction naturally captures the most relevant physical applications, as $\phi$ acts as a symmetric compensating symbol regulating the broader complex-valued symbol $\psi$. 
Despite this focus in the main results, we prove several intermediate lemmas for general complex-valued symbols to preserve mathematical generality.

It is worth noting the rich literature surrounding specific classes of complex-valued symbols. 
For instance, Walter Farkas, Niels Jacob, and René L. Schilling \cite{J1, FJS1} have made foundational contributions to the functional analytic framework of pseudo-differential operators by systematically developing the theory of function spaces generated by continuous negative definite symbols. Bridging probability theory—specifically the study of Lévy and Feller processes—with harmonic analysis, they rigorously established the properties of generalized Bessel potential spaces (often referred to as anisotropic fractional Sobolev spaces). 
Their comprehensive work proved essential structural characteristics of these spaces, including their complete Banach and Hilbert space formulations, Sobolev-type continuous embeddings, duality, and the critical density of the Schwartz class. (Note that while the literature typically refers to these as ``$\psi$-Bessel potential spaces," we utilize $\phi$ here to avoid confusion with the main symbol $\psi$ in \eqref{ab eqn}).

Crucially, our symbol $\phi$ is defined under much broader assumptions and does not, in general, correspond to the generator of a stochastic process. 
Consequently, the elegant structural properties of the function spaces derived in \cite{J1, FJS1} cannot be universally expected in our generalized setting. 
For example, while $\phi$ is assumed to be a non-negative function in one of our main results (Theorem \ref{main thm 3}), the sole mathematical requirement is simply that $\phi(\xi) \ge 0$ almost everywhere. 
Because this condition is so minimal, one cannot expect the resulting space to inherently possess nice properties as a function space. 
We urge the reader to carefully distinguish our simple non-negativity assumption from the highly structured ``continuous negative definite symbols" studied in \cite{J1, FJS1}.
\end{rem}

To clearly highlight our primary results and their underlying assumptions, we detail them in the subsequent section.

\mysection{Main result}

Let us first establish our assumptions regarding the weight function. 
In our framework, the weight depends solely on the temporal variable. 
This temporal weight plays a crucial role in constructing both spatial Fourier weak and strong solutions by mitigating the singularities of the symbols. 
Specifically, our symbols are permitted to blow up arbitrarily near the initial time. 
Consequently, classical Muckenhoupt weights are insufficient to control these singular behaviors. 
It is therefore essential to introduce a new class of weights tailored to accommodate this arbitrary, symbol-dependent blow-up. The detailed assumption for our weight is given below.

\begin{assumption}[Quasi-decreasing Weight]
										\label{weight as}
Let $w$ be a real-valued measurable function defined a.e. on the interval $(0,T)$. We assume that $w$ satisfies the following conditions:
\begin{enumerate}[(i)]
\item The weight is strictly positive almost everywhere:
\begin{align*}
w(t) >0 \quad \text{a.e.}~ t \in (0,T)
\end{align*}
\item There exists a constant $N_w > 0$ such that the weight is quasi-decreasing; that is,
\begin{align*}
\frac{w(t)}{w(s)} \leq N_w \quad \text{a.e.}~ 0<s<t<T.
\end{align*}
\end{enumerate}
\end{assumption}

\begin{rem}
A quasi-decreasing (or almost decreasing) function serves as a robust generalization of a classical non-increasing function. It is defined by the condition $w(t) \le C w(s)$ for all $s < t$, where $C \ge 1$ is an absolute constant. 
The core advantage of such a function is its ability to capture the macroscopic decay or blow-up of a standard decreasing function while simultaneously permitting localized, microscopic oscillations. 
Structurally, a powerful property of a quasi-decreasing function is its equivalence to a purely non-increasing function $\tilde{w}(t)$. Specifically, one can construct $\tilde{w}$ as either
$$
\tilde{w}(t) := \esssup_{t \le s < T} w(s) \quad \text{or} \quad \tilde{w}(t) := \essinf_{0 < s \le t} w(s),
$$
ensuring that $c_1 \tilde{w}(t) \le w(t) \le c_2 \tilde{w}(t)$ for some constants $c_1, c_2 > 0$. Thanks to this tight equivalence, quasi-decreasing weights seamlessly inherit the integrability and convergence properties of monotonic functions. 
This inherent stability under integral transformations, such as Hardy-type operators, makes them indispensable in the study of weighted spaces—particularly when addressing arbitrary endpoint singularities that classical Muckenhoupt $A_p$-weights fail to control. 
For further details, we refer the reader to \cite{KMP1,K1,BGT1,KMS1}.
\end{rem}

Recall that our assumption requires neither continuity nor the existence of a finite limit at $t=0$. 
Consequently, this framework accommodates the following mathematically interesting examples.
\begin{example}
\begin{itemize}
\item It naturally follows that every positive, non-increasing function $w$ on $(0,T)$ satisfies this assumption. In particular, our weight class includes functions such as $w(t) = \exp(1/t^\alpha)$ for $\alpha > 0$. This means the class permits weights that blow up exponentially near the initial time. In fact, because any non-increasing function is admissible, the weight can exhibit an arbitrarily severe blow-up as $t \to 0$.

\item Furthermore, highly fluctuating functions also satisfy the assumption, such as $w(t) = \frac{2 + \sin(1/t^\beta)}{t^\alpha}$ for positive constants $\alpha$ and $\beta$.

\item While our sine example $\sin(1/t^\beta)$ oscillates with rapidly increasing frequency as $t \to 0$, the class also includes weights that oscillate logarithmically:
$$
w(t) = \frac{1}{t^\alpha} \left( 2 + \cos\left( \ln\left( \frac{1}{t} \right) \right) \right) \quad \text{for } \alpha > 0
$$

\item There is no requirement for the weight to be continuous. 
We can explicitly construct a non-smooth weight with jumps that still belongs to our class:
$$
w(t) = \frac{1}{t^\alpha} \left( 1.5 + 0.5(-1)^{\lfloor 1/t \rfloor} \right) \quad \text{for } \alpha > 0
$$

\item A quasi-decreasing function need not blow up; it may remain bounded without converging:
$$
w(t) = 2 + \sin\left(\frac{1}{t^\beta}\right) \quad \text{for } \beta > 0
$$

\item We can combine different types of singularities and extremely high-frequency oscillations into a single quasi-decreasing weight:
$$
w(t) = \exp\left(\frac{1}{t}\right) \ln\left(1 + \frac{1}{t}\right) \left( 2 + \cos\left( \exp\left(\frac{1}{t}\right) \right) \right)
$$

\end{itemize}
\end{example}

Recall that the symbol $\psi(t,\xi)$ is a complex-valued function defined almost everywhere on the domain $(0,T) \times \mathbb{R}^d$. 
We now detail the primary assumptions imposed on $\psi(t,\xi)$, beginning with the non-positivity (or degenerate ellipticity) condition required of its real part.
\begin{assumption}[\bf Degenerate Ellipticity]
										\label{main as}
Assume that 
\begin{align}
									\label{elliptic con}
\Re[\psi(t,\xi)] \leq 0  \qquad a.e.~ (t,\xi)  \in (0,T) \times \bR^d.
\end{align}
\end{assumption}

\begin{assumption}[\bf Interior Local Integrability in Time]
										\label{main as 2}
Assume that for almost every $\xi \in \bR^d$, the symbol $\psi(t,\xi)$ is integrable over any compact subinterval of $(0,T)$ with respect to the temporal variable, meaning
$$
\int_s^t|\psi(r,\xi)| \mathrm{d}r < \infty \quad \text{for all } 0<s<t<T.
$$
\end{assumption}

Surprisingly, even under such weak conditions, a rigorous well-posedness theory can be established provided we seek a spectral-limit solution (Definition \ref{spectral-limit soluion}).

\begin{thm}
							\label{main thm}
Let $w$ be a weight satisfying Assumption \ref{weight as}, and suppose that Assumptions \ref{main as} and \ref{main as 2} hold. 
Then for any $f \in \mathrm{L}_{2,loc}\left( (0,T) \times \bR^d, w(t)\mathrm{d}t\mathrm{d}x \right)$, there exists a unique spectral-limit solution $u \in \mathrm{L}_{2,loc}\left( (0,T) \times \bR^d, w(t)\mathrm{d}t\mathrm{d}x \right)$ to equation \eqref{ab eqn}. Furthermore, this solution satisfies the following estimates:
\begin{align}
									\label{a priori est 2-2}
\int_0^t \int_{\bR^d} \left|u(s,x)\right|^2 \frac{w(s)}{s^2}\mathrm{d}x\mathrm{d}s
\leq 4N_w\int_0^t\int_{\bR^d}|f(s,x)|^2 w(s)\mathrm{d}x\mathrm{d}s,
\end{align}
and
\begin{align}
									\label{a priori est 2-3}
\esssup_{s \in [0,t]} \left( \|u(s,\cdot)\|^2_{\mathrm{L}_2(\bR^d)} \frac{w(s)}{s} \right)
\leq  N_w\int_0^{t} \int_{\bR^d} |f(s,x)|^2  w(s)  \mathrm{d}x\mathrm{d}s
\end{align}
for all $t \in (0,T)$.
Moreover, the solution $u$ can be represented as
\begin{align}
									\label{unique solution}
u(t,x)=  \cF^{-1}\left[ \int_0^t  \exp\left(\int_s^t\psi(r,\cdot)\mathrm{d}r \right) \cF[f(s,\cdot)]\mathrm{d}s \right](x).
\end{align}
\end{thm}

We provide the proof of Theorem \ref{main thm} in Section \ref{pf main thm}.

\begin{rem}
None of the estimates established in Theorem \ref{main thm} depend on any upper bound for the symbol $\psi$. 
This independence facilitates a specific approximation approach, even for solutions involving unbounded symbols. 
While the symbol is permitted to blow up arbitrarily in the temporal variable, Assumption \ref{main as 2} restricts this behavior to neighborhoods of the initial and terminal times. 
Conversely, blow-up with respect to the frequency variable $\xi$ is completely unrestricted. 
Furthermore, we do not require the symbol $\psi(t,\xi)$ to be continuous at any point in either the temporal or frequency domains.
\end{rem}

Assumption \ref{main as 2} alone is insufficient to interpret \eqref{ab eqn} in our weak sense using test functions. 
It is therefore natural to seek an additional condition on the symbol that permits such a weak formulation.

To establish stronger temporal regularity for the solution $u$, the most intuitive approach might be to impose stricter integrability conditions on the symbol with respect to the time variable. 
Surprisingly, however, this is not strictly necessary near $t=0$. 
Because the quasi-decreasing weight function we employ is sufficiently general to capture virtually any type of initial growth, we can still allow the symbol to exhibit arbitrary blow-up behavior as it approaches the initial time. 
The precise conditions utilized in this framework are detailed below.

\begin{assumption}[\textbf{local integrability}]
										\label{main as 3}
Let $w$ be a weight satisfying Assumption \ref{weight as}. We assume that the symbol $\psi(t,\xi)$ satisfies the following local square-integrability condition on $(0,T) \times \bR^d$:
\begin{align*}
\int_0^t \int_{B_R} \left(|\psi(s,\xi)|+|\psi(s,\xi)|^2 \frac{s^2}{w(s)} \right)\mathrm{d}\xi \mathrm{d}s < \infty \quad \text{for all } t \in (0,T) \text{ and } R > 0.
\end{align*}
\end{assumption}

\begin{thm}
							\label{main thm 2}
Suppose that Assumptions \ref{main as} and \ref{main as 3} hold, and let $f \in \mathrm{L}_{2,loc}\left( (0,T) \times \bR^d, w(t)\mathrm{d}t\mathrm{d}x \right)$. Then there exists a spatial Fourier weak solution $u \in \mathrm{L}_{2,loc}\left( (0,T) \times \bR^d, w(t)\mathrm{d}t\mathrm{d}x \right)$ to equation \eqref{ab eqn} satisfying \eqref{a priori est 2-2} and \eqref{a priori est 2-3}. 
This solution is unique among spatial Fourier weak solutions satisfying \eqref{a priori est 2-2}--\eqref{a priori est 2-3}, and it is explicitly given by \eqref{unique solution}.
\end{thm}

The proof of Theorem \ref{main thm 2} can also be found in Section \ref{pf main thm}.
\vspace{2mm}

Next, we introduce a condition stronger than Assumption \ref{main as 3}. 
This strengthened assumption enables us to establish a maximal regularity estimate and, subsequently, to construct a strong solution using an approximation argument based on this estimate.

\begin{assumption}[\textbf{(Degenerate) Ellipticity with respect to a spectral function $\phi$}]
									\label{main as 4}
Let $\phi$ be a non-negative function defined a.e.\ on $\bR^d$. 
We assume that the real part of the symbol satisfies the bound
\begin{align}
									\label{elliptic con phi}
\Re[\psi(t,\xi)] \leq -\phi(\xi) \quad \text{for a.e. } (t,\xi) \in (0,T) \times \bR^d.
\end{align}
\end{assumption}

Here, the term ``spectral function'' $\phi(\xi)$ refers to a purely frequency-dependent profile that governs the coercivity of the symbol in Fourier space. 
Because the only mathematical requirement imposed on $\phi$ is non-negativity, the choice of this function retains a great deal of flexibility.

\begin{assumption}[\textbf{Spatial compatibility with the spectral function $\phi$}]
										\label{main as 5}
Let $w$ be a weight satisfying Assumption \ref{weight as}, and suppose that Assumption \ref{main as 4} holds for a non-negative function $\phi$. We assume that
\begin{align}
										\label{20260403 10}
\int_0^t \esssup_{\xi \in \bR^d} \left( \frac{|\psi(s,\xi)|}{\phi(\xi)} \right)^2 \frac{1}{w(s)} \mathrm{d}s < \infty \quad \text{for all } t \in (0,T).
\end{align}
\end{assumption}

At first glance, requiring $\phi$ to be strictly positive in Assumption \ref{main as 4} appears necessary to prevent singularities in the ratio within Assumption \ref{main as 5}. 
However, because the principal symbol $\psi(s,\xi)$ can degenerate at any point, 
$\phi(\xi)$ may also vanish wherever $\psi(\xi) = 0$, provided we set $\frac{0}{0} := 0$ in \eqref{20260403 10}. 
This convention perfectly aligns with the assumption's core objective, which is to guarantee that
$$
|\psi(t,\xi)| = \frac{|\psi(t,\xi)|}{\phi(\xi)} \phi(\xi)
$$
across all $t \in (0,T)$ and $\xi \in \bR^d$. 
Furthermore, as shown in Remarks \ref{phi zero lemma} and \ref{degenerate operator}, this relaxation does not affect the resulting maximal regularity estimates and continuity of operators.

\begin{thm}
							\label{main thm 3}
Suppose that Assumptions \ref{main as 3} and \ref{main as 5} hold, and let $f \in \mathrm{L}_{2,loc}\left( (0,T) \times \bR^d, w(t)\mathrm{d}t\mathrm{d}x \right)$. 
Then there exists a unique strong solution $u$ to equation \eqref{ab eqn} belonging to the class
$$
\mathrm{L}_{\infty,2, loc}\left( (0,T) \times \bR^d ,\frac{w(t)}{t}\mathrm{d}t \mathrm{d}x\right) \cap \mathrm{L}_{2, loc}\left( (0,T) \times \bR^d,\frac{w(t)}{t^2}\mathrm{d}t\mathrm{d}x \right) \cap \mathrm{H}_{2,loc}^{\phi}\left( (0,T) \times \bR^d, w(t)\mathrm{d}t\mathrm{d}x \right).
$$
This solution satisfies the estimates \eqref{a priori est 2-2} and \eqref{a priori est 2-3}. In addition, for every $t \in (0,T)$, it satisfies
\begin{align}
									\notag
&\int_0^t \left( \int_{\bR^d}  \left| \psi(s,-\mathrm{i}\nabla) u(s,x) \right|^2 \mathrm{d}x  \right)^{1/2} \mathrm{d}s \\
									\label{a priori est 00}
&\leq  \left(\int_0^t  \esssup_{\xi \in \bR^d} \left(\left| \frac{\psi(s,\xi)}{\phi(\xi)} \right|\right)^2 \frac{1}{w(s)} \mathrm{d}s \right)^{1/2} \left(\int_0^t \int_{\bR^d} \left|\phi(-\mathrm{i}\nabla)u(s,x)\right|^2 w(s)\mathrm{d}x\mathrm{d}s\right)^{1/2}
\end{align}
and
\begin{align}
 \int_0^t \int_{\bR^d} \left|\phi(-\mathrm{i}\nabla)u(s,x)\right|^2 w(s)\mathrm{d}x\mathrm{d}s 
									\label{a priori est}
&\leq N_w\int_0^t \int_{\bR^d}|f(s,x)|^2 w(s) \mathrm{d}x\mathrm{d}s.
\end{align}

Moreover, the unique solution $u$ is explicitly given by \eqref{unique solution}.

\end{thm}

The proof of Theorem \ref{main thm 3} is similarly deferred to Section \ref{pf main thm}.

\begin{rem}
\begin{itemize}
\item 
By combining inequalities \eqref{a priori est 00} and \eqref{a priori est}, we obtain the following bound:
$$
\begin{aligned}
&\int_0^t \left( \int_{\bR^d}  \left| \psi(s,-\mathrm{i}\nabla) u(s,x) \right|^2 \mathrm{d}x  \right)^{1/2} \mathrm{d}s \\
&\quad\leq  \sqrt{N_w}\left(\int_0^t  \esssup_{\xi \in \bR^d} \left(\left| \frac{\psi(s,\xi)}{\phi(\xi)} \right|\right)^2 \frac{1}{w(s)} \mathrm{d}s \right)^{1/2}
\left(\int_0^t\int_{\bR^d}|f(s,x)|^2 w(s) \mathrm{d}x\mathrm{d}s\right)^{1/2}.
\end{aligned}
$$
This estimate shows that the symbols do not need to be bounded in the temporal variable for a strong solution to exist. 
Instead, weighted local integrability of the symbols is sufficient to guarantee the existence of a unique strong solution.

When studying general PDEs involving non-linearities or variable coefficients, perturbation methods serve as a fundamental tool. Because these techniques generally rely on the availability of strong solutions, PDE applications frequently assume bounded coefficients to ensure such solutions can be readily found. 
However, as demonstrated by our theorem, this strict boundedness assumption can be successfully relaxed to allow for coefficients that are merely locally square $(1/w)$-integrable.

\item 

For convenience, we introduce the space $\mathrm{H}_{1,2,\mathrm{loc}}^{\psi}\left( (0,T) \times \bR^d\right)$ to denote the class of functions $u$ satisfying
\begin{align}
									\label{psi space}
\int_0^t \left( \int_{\bR^d}  \left| \psi(s,-\mathrm{i}\nabla) u(s,x) \right|^2 \mathrm{d}x  \right)^{1/2} \mathrm{d}s < \infty \quad \forall t \in (0,T).
\end{align}
In view of \eqref{a priori est 00}, provided that \eqref{20260403 10} holds, one readily obtains the continuous inclusion
\begin{align*}
\mathrm{H}_{2,loc}^{\phi}\left( (0,T) \times \bR^d, w(t)\mathrm{d}t\mathrm{d}x \right)
\subset 
\mathrm{H}_{1,2,loc}^{\psi}\left( (0,T) \times \bR^d \right)
\end{align*}
Naturally, establishing solution uniqueness within a broader function class yields a stronger and more desirable result. 
As presented in Theorem \ref{uniqueness theorem}, we later demonstrate that the uniqueness indeed holds within this larger space.

\item

The constant in \eqref{a priori est} is the quasi-decreasing constant $N_w$ and is independent of the spatial dimension $d$. 
This dimension-free characteristic also occurs in certain $L_p$-estimates for $p \in (1,\infty)$ when evaluating the $L_p$-norm of the second-order directional derivative. 
Although classical Calderón-Zygmund theory is insufficient to derive this independence, probabilistic methodologies offer a rigorous proof for a particular $L_p$-estimate (as demonstrated in \cite{KP1}).

\item

Furthermore, this framework clarifies the relationship between the main operator and the uniform ellipticity condition. 
For instance, consider the operator $\psi(t, -\mathrm{i}\nabla)u = a^{ij}(t)u_{x^ix^j}$ satisfying uniform ellipticity, meaning there exists a positive constant $\kappa$ such that
$$
a^{ij}(t) \xi^i\xi^j \geq \kappa |\xi|^2 \quad \text{for a.e. } t \in (0,T) \text{ and } \forall \xi \in \bR^d.
$$
By setting $\phi(\xi) = \kappa |\xi|^2$ in \eqref{a priori est}, we can derive the following inequality for the Laplacian:
$$
\int_0^T \int_{\bR^d} \left| (\Delta u) (t,x)\right|^2 w(t)\mathrm{d}x\mathrm{d}t \leq \frac{N_w}{\kappa^2} \int_0^T\int_{\bR^d}|f(t,x)|^2 w(t) \mathrm{d}x\mathrm{d}t.
$$
Notably, the fact that the constant depends exclusively and explicitly on the ellipticity parameter remains valid for general values of $p \in (1,\infty)$. 
Much like the property of dimension independence, deriving this relationship is challenging if one relies solely on the $L_p$-boundedness of singular integral operators. 
Instead, it is most effectively established using probabilistic techniques (see \cite{KID 2024, KK 2023}).

\item The functional class in Theorem \ref{main thm 3} establishes a subtle balance between time-weight decay near $t=0$ and spatial spectral regularity. 
Specifically, the factor $t^{-2}$ in $L_{2,loc}(w(t)/t^2 dtdx)$ quantifies the precise rate at which the solution vanishes as $t \to 0^+$, compensating for the lack of operator integrability, while $H_{2,loc}^\phi(w(t)dtdx)$ captures the optimal spatial regularity governed by the spectral function $\phi(\xi)$.

\end{itemize}
\end{rem}

\begin{rem}
													\label{zero initial reason}

We emphasize that handling non-zero initial conditions presents a fundamental challenge in this framework. 
Formally, one can write the solution $u$ to the initial-value problem
\begin{equation*}
\begin{cases}
\partial_tu(t,x)=\psi(t,-\mathrm{i}\nabla)u(t,x),\quad &(t,x)\in(0,T)\times\mathbb{R}^d,\\
u(0,x)=u_0(x),\quad & x\in\mathbb{R}^d,
\end{cases}
\end{equation*}
via the Fourier representation
$$  
u(t,x) = \cF^{-1}\left[ \exp\left(\int_0^t\psi(r,\cdot)\mathrm{d}r \right) \cF[u_0] \right].  
$$
Nevertheless, since we do not assume the symbol $\psi(t,\xi)$ is locally integrable near $t=0$, the time integral $\int_0^t\psi(r,\xi)\mathrm{d}r$ generally fails to make sense, even when interpreted as a tempered distribution. 
Therefore, incorporating non-zero initial data strictly requires imposing an extra temporal integrability condition on $\psi(t,\xi)$. 
This restriction remains essential even for second-order operators when their coefficients exhibit severe singularities near the initial time $t=0$ (see \cite[Remark 2.34]{KK 2026}). 
Indeed, as discussed in \cite[Remark 2.34]{KK 2026} for second-order equations with non-integrable time coefficients, extreme singularities near $t=0$ structurally prevent solutions from converging to arbitrary non-zero initial data $u_0$ in standard unweighted spaces.

In summary, our framework guarantees well-posedness for the inhomogeneous problem with zero initial data ($u(0,x)=0$); the resulting solution decays naturally under the given weighted norms (such as $L_{2,loc}(w(t)/t^2 dtdx)$). 
Conversely, admitting general non-zero initial data ($u_0 \neq 0$) degrades well-posedness unless $\psi(t,\xi)$ satisfies stronger integrability conditions as $t \to 0^+$.

\end{rem}

\mysection{Applications to PDEs (A toy model)}

We believe that our main theorem paves the way for many interesting new results, even within the context of classical second-order equations. However, to fully demonstrate the broad generality of our framework, we have chosen to present toy models featuring operators of an arbitrary real order. At the same time, we restrict our attention to isotropic operators to minimize unnecessary technical complexities. For the reader's convenience and to make the results in this section more readily accessible, we now recall equation \eqref{ab eqn 2}:
\begin{equation}
									\label{main eqn 2}
\begin{cases}
\partial_tu(t,x)= -t^\alpha \exp\left(  t^\beta \right) (-\Delta)^{\gamma/2} \exp\left( (-\Delta)^{\delta/2} \right)  u(t,x)+f(t,x),\quad &(t,x)\in(0,T)\times\mathbb{R}^d,\\
u(0,x)=0,\quad & x\in\mathbb{R}^d,
\end{cases}
\end{equation}
where
\begin{align*}
-t^\alpha \exp\left(  t^\beta \right) (-\Delta)^{\gamma/2} \exp\left( (-\Delta)^{\delta/2} \right)  u(t,x)
=\cF^{-1}\left[\psi(t,\cdot)\cF[u](t,\cdot)\right](x)
\end{align*}
and
$$
\psi(t,\xi)=-t^\alpha \exp\left( t^\beta \right) |\xi|^{\gamma} \exp\left( |\xi|^{\delta} \right). 
$$

\begin{thm}
									\label{main appl 1}
Let $\alpha, \beta, \gamma, \delta$ be real numbers, and suppose $f \in \mathrm{L}_{2,loc}\left( (0,T) \times \bR^d, w(t)\mathrm{d}t\mathrm{d}x \right)$, where $w$ is a weight satisfying Assumption \ref{weight as}. Then there exists a unique spectral-limit solution 
$$
u \in \mathrm{L}_{2,loc}\left( (0,T) \times \bR^d, w(t)\mathrm{d}t\mathrm{d}x \right)
$$
to equation \eqref{main eqn 2} such that \eqref{a priori est 2-2} and \eqref{a priori est 2-3} hold.
\end{thm}

\begin{proof}
It suffices to verify that the hypotheses of Theorem \ref{main thm} are satisfied for the symbol
$$
\psi(t,\xi)=-t^\alpha \exp\left( t^\beta \right) |\xi|^{\gamma} \exp\left( |\xi|^{\delta} \right).
$$
More precisely, we only need to confirm that Assumptions \ref{main as} and \ref{main as 2} hold. This verification is straightforward. First, the real part of the symbol trivially satisfies the required upper bound:
$$
\Re[\psi(t,\xi)] = \psi(t,\xi) = -t^\alpha \exp\left( t^\beta \right) |\xi|^{\gamma} \exp\left( |\xi|^{\delta} \right) \leq 0 \quad \text{for a.e. } (t,\xi) \in (0,T) \times \bR^d.
$$
Second, the local integrability condition,
$$
\int_s^t|\psi(r,\xi)| \mathrm{d}r < \infty \quad \text{for all } 0<s<t<T \text{ and a.e. } \xi \in \bR^d,
$$
is clearly met. 
This is because the temporal coefficient $r^\alpha \exp\left( r^\beta \right)$ is continuous on $(0,T)$ and can only locally blow up near the initial time $r=0$, regardless of whether $\alpha$ and $\beta$ are negative.

\end{proof}

We now turn our attention to the existence of a spatial Fourier weak solution to equation \eqref{main eqn 2}. 
To ensure that the integrals in the weak formulation are well-defined, 
we must restrict the range of certain exponents to guarantee local integrability.

While our previous analysis focused on a general quasi-decreasing weight $w$, establishing more regular solutions—such as spatial Fourier weak solutions and strong solutions—requires us to impose specific conditions on the weight. 
This restriction is mathematically natural; when the principal operators exhibit rapid singularities, strong solutions cannot exist unless they are appropriately compensated by a specialized weight.

\begin{thm}
									\label{main appl 2}
Let $\alpha \in (-1,\infty)$, $\beta=0$, $\gamma \in (-d/2,\infty)$, $\delta \in [0,\infty)$, and $\theta \in (-\infty, 3/2) \cap (-\infty,-\alpha]$. Define the weight
$$
w(t)=t^{2\alpha + 2\theta},
$$
and suppose $f \in \mathrm{L}_{2,loc}\left( (0,T) \times \bR^d, w(t)\mathrm{d}t\mathrm{d}x \right)$. 
Then there exists a spatial Fourier weak solution $u \in \mathrm{L}_{2,loc}\left( (0,T) \times \bR^d, w(t)\mathrm{d}t\mathrm{d}x \right)$ to equation \eqref{main eqn 2} satisfying \eqref{a priori est 2-2} and \eqref{a priori est 2-3}. It is unique among spatial Fourier weak solutions satisfying these two estimates.
\end{thm}

\begin{proof}
To apply Theorem \ref{main thm 2}, it suffices to verify its hypotheses under the specified parameter ranges. 
Since Assumption \ref{main as} is trivially satisfied, our task reduces to confirming Assumptions \ref{weight as} and \ref{main as 3}.

To confirm Assumption \ref{weight as}, we restrict the exponents so that the weight $w(t)$ is monotonically decreasing. 
This requires $\alpha + \theta \leq 0$. 

Next, we verify the strong local integrability condition of Assumption \ref{main as 3}:
\begin{align*}
&\int_0^t \int_{B_R} \left(|\psi(s,\xi)|+|\psi(s,\xi)|^2 \frac{s^2}{w(s)} \right)\mathrm{d}\xi \mathrm{d}s \\
&\simeq\int_0^t \int_{B_R} \left(s^\alpha  |\xi|^{\gamma} \exp\left( |\xi|^{\delta} \right)+ \left( |\xi|^{\gamma} \exp\left( |\xi|^{\delta} \right) \right)^2 s^{2-2\theta} \right)\mathrm{d}\xi \mathrm{d}s < \infty 
\end{align*}
for all $t \in (0,T)$ and $R \in (0,\infty)$. This integrability requirement naturally imposes restrictions on the exponents: we need $\gamma > -d/2$ for spatial integrability near the origin, $\delta \geq 0$ for appropriate exponential behavior, and $\alpha > -1$, $\theta < 3/2$ for temporal integrability near $s=0$. 

Combining these constraints yields $\gamma \in (-d/2,\infty)$, $\delta \in [0,\infty)$, and 
$\theta \in (-\infty, 3/2) \cap (-\infty,-\alpha]$. Under these conditions, Assumptions \ref{weight as} and \ref{main as 3} are fully satisfied, completing the proof.
\end{proof}

Finally, we turn our attention to establishing the existence of a unique strong solution to equation \eqref{main eqn 2}. 
In this context, we restrict our analysis to a finite time interval $(0,T)$. 
However, this restriction does not limit the scope of our theory; since $T>0$ can be chosen arbitrarily large, obtaining a unique strong solution on $(0,T)$ ultimately implies the existence of a unique global strong solution on the entire temporal interval $(0,\infty)$.

\begin{thm}
									\label{main appl 3}
Let $T \in (0,\infty)$, $\alpha\in (-1,0]$, $\beta=0$, $\gamma \in (-d/2,\infty)$, $\delta \in [0,\infty)$, and $\theta \in  (-\infty, 1/2) \cap (-\infty,-\alpha]$. 
Define the weight $w$ and the spectral function $\phi$ by
$$
w(t)=t^{2\alpha+2\theta}  \quad \text{and} \quad \phi(\xi)= T^\alpha |\xi|^{\gamma} \exp\left( |\xi|^{\delta} \right),
$$
respectively. 
For any $f \in \mathrm{L}_{2,loc}\left( (0,T) \times \bR^d, w(t)\mathrm{d}t\mathrm{d}x \right)$, equation \eqref{main eqn 2} admits a unique strong solution $u$ associated with the spectral function $\phi$ in the class
$$
\mathrm{L}_{\infty,2, loc}\left( (0,T) \times \bR^d ,\frac{w(t)}{t}\mathrm{d}t \mathrm{d}x\right) \cap \mathrm{L}_{2, loc}\left( (0,T) \times \bR^d,\frac{w(t)}{t^2}\mathrm{d}t\mathrm{d}x \right) \cap \mathrm{H}_{2,loc}^{\phi}\left( (0,T) \times \bR^d, w(t)\mathrm{d}t\mathrm{d}x \right).
$$
Furthermore, this solution satisfies the estimates \eqref{a priori est 2-2}, \eqref{a priori est 2-3}, \eqref{a priori est 00}, and \eqref{a priori est}.
\end{thm}

\begin{proof}
To invoke Theorem \ref{main thm 3}, we must demonstrate that Assumptions \ref{main as 3} and \ref{main as 5} hold under the stated hypotheses. This entails verifying four conditions: first, that the weight $w(t)$ is monotonically decreasing; second, that the symbol satisfies the ellipticity bound
\begin{align}
										\label{20260329 10}
\psi(t,\xi) \leq - \phi(\xi) \quad \text{for a.e. } (t,\xi) \in (0,T) \times \bR^d;
\end{align}
third, that the strong local integrability condition is satisfied,
\begin{align*}
\int_0^t \int_{B_R} |\psi(s,\xi)|^2 \frac{s^2}{w(s)} \mathrm{d}\xi \mathrm{d}s < \infty \quad \text{for all } t \in (0,T) \text{ and } R \in (0,\infty);
\end{align*}
and finally, that the symbol is compatible with the spectral function $\phi$, meaning
\begin{align}
										\label{20260327 01} 
\int_0^t \esssup_{\xi \in \bR^d} \left( \frac{|\psi(s,\xi)|}{\phi(\xi)} \right)^2 \frac{1}{w(s)} \mathrm{d}s < \infty \quad \text{for all } t \in (0,T).
\end{align}

Drawing on the proof of the preceding theorem, the requirements for a decreasing weight $w$ and the local integrability condition (Assumption \ref{main as 3}) immediately restrict our parameters to $\gamma \in (-d/2,\infty)$, $\delta \in [0,\infty)$, and $\theta \in (-\infty, 3/2) \cap (-\infty,-\alpha]$. 

For the ellipticity bound \eqref{20260329 10}, we note that since $\alpha \leq 0$, the temporal coefficient $t^\alpha$ is a non-increasing function on the interval $(0,T]$. Consequently, it attains its infimum at $t=T$, yielding the lower bound $T^\alpha $ and thereby validating \eqref{20260329 10}.

Lastly, we check the compatibility condition \eqref{20260327 01} by explicitly computing the integral:
\begin{align*}
\int_0^t \esssup_{\xi \in \bR^d} \left( \frac{|\psi(s,\xi)|}{\phi(\xi)} \right)^2 \frac{1}{w(s)} \mathrm{d}s 
&\lesssim \int_0^t s^{-2\theta} \mathrm{d}s < \infty.
\end{align*}
Convergence at the origin $s=0$ dictates that $-2\theta > -1$, which forces $\theta < 1/2$. Intersecting all of these necessary constraints yields the exact parameter ranges specified in the theorem, which completes the proof.
\end{proof}

\begin{rem}
Consider the case where $\gamma = 2$ and $\alpha=\beta=\delta=0$, under which equation \eqref{main eqn 2} simplifies to the classical heat equation. Even in this well-studied setting, our findings provide a novel contribution. 
In this case, the weight in Theorem \ref{main appl 3} is $w(t)=t^{2\theta}$. It belongs to the one-dimensional Muckenhoupt class $A_2(\bR)$ if and only if $\theta\in(-1/2,1/2)$. 
Consequently, the range $\theta\leq-1/2$ lies outside the classical $A_2$ regime and illustrates the strongly singular weights covered by our framework.

Conversely, our framework for highly singular weights does not capture the entirety of Muckenhoupt weights. 
For the classical heat equation, standard weighted $\mathrm{L}_2$-theory applies to $t^{2\theta}$ throughout $\theta\in(-1/2,1/2)$, whereas the monotonicity condition in our toy model with $\alpha=0$ restricts us to $\theta\leq0$. 
For comprehensive treatments of the heat equation and its generalizations involving temporal $A_2$-weights, we refer the reader to \cite{PS1,MV1,CJH KID 2023,DK2,DK3}.
\end{rem}

\begin{rem}

While Model \eqref{main eqn 2} focuses on isotropic operators for expositional clarity, our general framework (Theorems \ref{main thm}, \ref{main thm 2}, and \ref{main thm 3}) applies equally to complex-valued symbols possessing highly singular imaginary parts. For instance, consider the Cauchy problem driven by a singular advection-diffusion operator:
  $$
  \partial_t u(t,x) = -t^{\alpha}(-\Delta)^{\gamma/2}u(t,x) +  t^{-\beta} b \cdot \nabla u(t,x) + f(t,x),
  $$
where $b \in \mathbb{R}^d \setminus \{0\}$. Here, the corresponding symbol is given by $\psi(t,\xi) = -t^\alpha \vert{}\xi\vert{}^\gamma + it^{-\beta}(b \cdot \xi)$. 
Since $\Re[\psi(t,\xi)] = -t^\alpha \vert{}\xi\vert{}^\gamma \le 0$, Assumption \ref{main as} holds identically regardless of the severe temporal blow-up (for any $\beta>0$) or high-frequency oscillations induced by the purely imaginary transport term $\Im[\psi(t,\xi)] = t^{-\beta}(b \cdot \xi)$ near $t=0$. 
Thus, the presence of severe temporal oscillations or singular transport terms does not disrupt the well-posedness of spectral-limit solutions.
\end{rem}

\mysection{A priori estimates}

Deriving \textit{a priori} estimates, frequently alongside the method of continuity, is a standard and powerful technique for establishing the well-posedness of PDEs. 
However, because the symbols in our operators are unbounded, the method of continuity is inapplicable. 
Furthermore, while the traditional approach assumes a solution $u$ exists before deriving its bounds, we adopt a different strategy. 
We explicitly construct a candidate solution to \eqref{ab eqn} and establish estimates directly for this specific function. 
Since uniqueness holds within our designated function classes, this strategy is conceptually equivalent to classical well-posedness arguments based on standard \textit{a priori} estimates and the method of continuity.
 Consequently, we continue to refer to our derived bounds as ``\textit{a priori} estimates," despite them originating from a concrete candidate. 
Specifically, we define our candidate solution to \eqref{ab eqn} as
\begin{align}
										\label{u defn}
u(t,x)=  \cF^{-1}\left[ \int_0^t  \exp\left(\int_s^t\psi(r,\cdot)\mathrm{d}r \right) \cF[f(s,\cdot)]\mathrm{d}s \right](x)
\end{align}
and dedicate the remainder of this section to establishing fundamental estimates for this explicitly defined $u$.

In other words, we derive \textit{a priori} estimates for the function $u$ in \eqref{u defn}, provisionally treating it as a solution to \eqref{ab eqn}. 
The rigorous justification of $u$ as a solution will be provided in the next section.
These \textit{a priori} estimates require no upper bound on $\psi(t,\xi)$; this fact motivated our main theorem.

\begin{lem}
							\label{l2 bounded lemma 1}
Assume that $f \in \mathrm{L}_{2,loc}\left( (0,T) \times \mathbb{R}^d \right)$ and let $u$ be given by \eqref{u defn}. 
Under the conditions specified in Assumptions \ref{main as} and \ref{main as 2}, 
the following inequalities hold for almost every $t \in (0,T)$:
\begin{align}
										\label{20260321 01}
\|u(t,\cdot)\|^2_{\mathrm{L}_2(\mathbb{R}^d)} 
\leq \int_{\mathbb{R}^d} \left|\int_0^t |\mathcal{F}[f(s,\cdot)](\xi)| \mathrm{d}s \right|^2 \mathrm{d}\xi,
\end{align}
\begin{align}
										\label{20260321 01-2}
\esssup_{s \in [0,t]} \left( \|u(s,\cdot)\|^2_{\mathrm{L}_2(\bR^d)} \frac{1}{s} \right)
\leq  \int_0^{t} \int_{\bR^d} |f(s,x)|^2   \mathrm{d}x\mathrm{d}s,
\end{align}
and
\begin{align}
										\label{20260321 02}
\int_0^t \|u(s,\cdot)\|^2_{\mathrm{L}_2(\bR^d)} \frac{1}{s^2} \mathrm{d}s
\leq   4\int_0^t \int_{\bR^d}  \left|f(s,x)\right|^2 \mathrm{d}x \mathrm{d}s.
\end{align}
\end{lem}

\begin{proof}

We begin by deriving the initial bound, \eqref{20260321 01}. 
By employing Plancherel's theorem and utilizing the fact that the real part of the symbol is non-positive,
 we deduce
 $$
\|u(t,\cdot)\|^2_{\mathrm{L}_2(\mathbb{R}^d)} \leq \int_{\mathbb{R}^d} \left|\int_0^t |\mathcal{F}[f(s,\cdot)](\xi)| \mathrm{d}s \right|^2 \mathrm{d}\xi.
 $$
 This establishes the first bound. 
 Furthermore, applying Jensen's inequality to \eqref{20260321 01} yields
\begin{align}
										\label{20260405 01}
\|u(t,\cdot)\|^2_{\mathrm{L}_2(\mathbb{R}^d)} 
\leq \int_{\mathbb{R}^d} \left|\int_0^t |\mathcal{F}[f(s,\cdot)](\xi)| \mathrm{d}s \right|^2 \mathrm{d}\xi 
\leq t\int_0^t \int_{\mathbb{R}^d} |f(s,x)|^2 \mathrm{d}x \mathrm{d}s.
\end{align}
 The second estimate, \eqref{20260321 01-2}, follows immediately by dividing \eqref{20260405 01} by $t$ and taking the essential supremum. 
 To deduce the final inequality, we multiply the first two terms of \eqref{20260405 01} by $t^{-2}$, which results in
 $$
 \frac{1}{t^2}\|u(t,\cdot)\|^2_{\mathrm{L}_2(\mathbb{R}^d)} \leq \int_{\mathbb{R}^d} \left|\frac{1}{t}\int_0^t |\mathcal{F}[f(s,\cdot)](\xi)| \mathrm{d}s \right|^2 \mathrm{d}\xi.
 $$
 Integrating this expression with respect to the temporal variable, we find
 $$
 \int_0^{T'} \frac{1}{t^2}\|u(t,\cdot)\|^2_{\mathrm{L}_2(\mathbb{R}^d)} \mathrm{d}t \leq \int_0^{T'}\int_{\mathbb{R}^d} \left|\frac{1}{t}\int_0^t |\mathcal{F}[f(s,\cdot)](\xi)| \mathrm{d}s \right|^2 \mathrm{d}\xi\mathrm{d}t~ \quad \forall T' \in (0,T).
 $$
 Finally, invoking Tonelli's theorem, the classical Hardy inequality, and Plancherel's theorem completes the proof, yielding \eqref{20260321 02}.
\end{proof}

It is straightforward to extend the preceding lemma to cases involving a quasi-decreasing weight $w$.

\begin{corollary}
							\label{l2 bounded corollary 1}
Let $f \in \mathrm{L}_{2,loc}\left( (0,T) \times \mathbb{R}^d, w(t)\mathrm{d}t\mathrm{d}x \right)$, and let $w$ be a weight function that satisfies Assumption \ref{weight as}. 
Given a function $u$ defined by \eqref{u defn}, and assuming that Assumptions \ref{main as} and \ref{main as 2} hold, the following estimates are valid for almost every $t \in (0,T)$:
\begin{align*}
\|u(t,\cdot)\|^2_{\mathrm{L}_2(\mathbb{R}^d)}  
\leq \int_{\mathbb{R}^d} \left|\int_0^t |\mathcal{F}[f(s,\cdot)](\xi)| \mathrm{d}s \right|^2 \mathrm{d}\xi,
\end{align*}
\begin{align}
										\label{20260321 03-2}
\esssup_{s \in [0,t]} \left( \|u(s,\cdot)\|^2_{\mathrm{L}_2(\bR^d)} \frac{w(s)}{s} \right)
\leq  N_w\int_0^{t} \int_{\bR^d} |f(s,x)|^2  \mathrm{d}x  w(s)  \mathrm{d}s,
\end{align}
and
\begin{align}
									\label{20260321 04}
\int_0^{t} \|u(s,\cdot)\|^2_{\mathrm{L}_2(\bR^d)}\frac{w(s)}{s^2}  \mathrm{d}s
&\leq 4 N_w\int_0^{t} \int_{\bR^d} \left|f(s,x) \right|^2 \mathrm{d}x w(s)\mathrm{d}s.
\end{align}
\end{corollary}
\begin{proof}
Due to Assumption \ref{weight as}, the unweighted locally square-integrable space contains the weighted one:
$$
\mathrm{L}_{2,loc}\left( (0,T) \times \mathbb{R}^d, w(t)\mathrm{d}t\mathrm{d}x \right) \subset \mathrm{L}_{2,loc}\left( (0,T) \times \mathbb{R}^d \right).
$$
As a result, the bound from \eqref{20260321 01} derived in the previous lemma remains applicable. 
By multiplying \eqref{20260321 01} by $w(t)$ and utilizing both Jensen's inequality and Assumption \ref{weight as}, we find that for almost every $t \in (0,T)$,
\begin{align}
										\label{20260405 10}
\frac{w(t)}{t} \|u(t,\cdot)\|^2_{\mathrm{L}_2(\mathbb{R}^d)} 
&\leq t w(t)\int_{\mathbb{R}^d} \left| \frac{1}{t}\int_0^t |\mathcal{F}[f(s,\cdot)](\xi)| \mathrm{d}s \right|^2 \mathrm{d}\xi\\
										\notag
&\leq w(t)\int_0^t \int_{\mathbb{R}^d} |f(s,x)|^2 \mathrm{d}x \mathrm{d}s  
\leq N_w \int_0^t \int_{\mathbb{R}^d} |f(s,x)|^2 w(s) \mathrm{d}x \mathrm{d}s.
\end{align}
This establishes \eqref{20260321 03-2}. 
For the final estimate, we multiply \eqref{20260405 10} by $1/t$ and integrate over time on the interval $(0, T')$. 
Utilizing Assumption \ref{weight as} once more gives us the following for all $T' \in (0,T)$:
$$
\int_0^{T'} \frac{w(t)}{t^2} \|u(t,\cdot)\|^2_{\mathrm{L}_2(\mathbb{R}^d)} \mathrm{d}t 
\leq N_w \int_0^{T'} \int_{\mathbb{R}^d} \left| \frac{1}{t}\int_0^t |\mathcal{F}[f(s,\cdot)](\xi)| \sqrt{w(s)}\mathrm{d}s \right|^2 \mathrm{d}\xi \mathrm{d}t.
$$
Finally, combining Fubini's theorem, the classical Hardy inequality, and Plancherel's theorem yields the required bound in \eqref{20260321 04}.
\end{proof}

\begin{rem}
										\label{extra weight rem}
Lemma \ref{l2 bounded lemma 1} and Corollary \ref{l2 bounded corollary 1} are equivalent: the lemma implies the corollary, and setting $w(t)=1$ for $t \in (0,T)$ in the corollary recovers the lemma. 
This reflects classical weighted extrapolation, where bounding a Calderón-Zygmund operator $\cA$ for a single Muckenhoupt $A_2(\bR^d)$ weight (including the unweighted $w=1$ case) guarantees its boundedness for all $A_2(\bR^d)$ weights (\textit{cf.}\cite{Hy1,GR1,CMP1}).
In our context, if a weight $w$ under Assumption \ref{weight as} is bounded above by a positive constant $\kappa$, such that $w(t) \leq \kappa$ for almost all $t \in (0,T)$, then only the following inclusion is immediate:
$$
\mathrm{L}_{2,loc}\left( (0,T) \times \bR^d \right)
= \mathrm{L}_{2,loc}\left( (0,T) \times \bR^d, w(t)\mathrm{d}t\mathrm{d}x \right).
$$
Thus, a structural similarity to extrapolation theorems with Muckenhoupt weights becomes apparent. 
However, such an extrapolation argument is not expected to apply directly to our class of weights.
\end{rem}

We now establish our primary estimate, demonstrating a regularity gain for the solution derived from the equation. 
This estimate is controlled by the spectral function $\phi$, which ensures the (degenerate) ellipticity condition is satisfied. 
We want to emphasize the strength of this result: despite being derived using only elementary tools, the estimate is remarkably powerful due to the inclusion of a very general weight.

\begin{lem}
							\label{l2 bounded lemma}
Let $w$ be a weight satisfying Assumption \ref{weight as}, and let $f \in \mathrm{L}_{2,loc}\left( (0,T) \times \bR^d, w(t)\mathrm{d}t\mathrm{d}x \right)$. Suppose $u$ is given by \eqref{u defn}. If Assumptions \ref{main as 2} and \ref{main as 4} hold, then for any $T' \in (0,T)$, we have the following maximal regularity estimate:
\begin{align}
									\label{maximal regularity estimate}
\int_0^{T'} \int_{\bR^d} \left|\phi(-\mathrm{i}\nabla)u(t,x)\right|^2 w(t) \mathrm{d}x\mathrm{d}t \leq  N_w \int_0^{T'} \int_{\bR^d}|f(t,x)|^2 w(t) \mathrm{d}x\mathrm{d}t.
\end{align}
\end{lem}
\begin{proof}
Fix an arbitrary $T' \in (0,T)$. We first observe the following bounds for all $0<s<t<T$ and $\xi \in \bR^d$:
$$
\exp\left( -s\phi(\xi) \right) \leq 1 \quad \text{and} \quad w(t) \leq N_w w(t-s),
$$
where the second inequality is a direct consequence of Assumption \ref{weight as}.
Applying Tonelli's theorem, Plancherel's theorem, and the generalized Minkowski inequality, we can bound the norm as follows:
\begin{align*}
&\int_0^{T'} \int_{\bR^d} \left|\phi(-\mathrm{i}\nabla)u(t,x)\right|^2 w(t) \mathrm{d}x\mathrm{d}t \\
&= \int_0^{T'} \int_{\bR^d} \left|\int_0^t \exp\left( \int_s^t\psi(r,\xi) \mathrm{d}r \right) \phi(\xi) \cF[f(s,\cdot)](\xi) \mathrm{d}s\right|^2  w(t)\mathrm{d}\xi \mathrm{d}t \\
&\leq \int_{\bR^d} \int_0^{T'}  \left|\int_0^t \phi(\xi) \exp\left( -(t-s) \phi(\xi) \right)  |\cF[f(s,\cdot)](\xi)| \mathrm{d}s\right|^2 w(t) \mathrm{d}t \mathrm{d}\xi \\
&\leq N_w \int_{\bR^d}   \left(\int^\infty_{0}  \phi(\xi) \exp\left( -s \phi(\xi) \right) \left(\int_{0}^{T'}  1_{0<t-s}|\cF[f(t-s,\cdot)](\xi)|^2 w(t-s) \mathrm{d}t\right)^{1/2} \mathrm{d}s \right)^2 \mathrm{d}\xi \\
&\leq N_w \int_{\bR^d}   \left(\int^\infty_{0}  \phi(\xi) \exp\left( -s \phi(\xi) \right) \left(\int_{0}^{T'} |\cF[f(t,\cdot)](\xi)|^2 w(t) \mathrm{d}t\right)^{1/2} \mathrm{d}s \right)^2 \mathrm{d}\xi  \\
&\leq N_w \int_{\bR^d}   \left(\int^\infty_{0}  \phi(\xi) \exp\left( -s \phi(\xi) \right)  \mathrm{d}s \right)^2 
\left(\int_{0}^{T'}  |\cF[f(t,\cdot)](\xi)|^2 w(t) \mathrm{d}t \right) \mathrm{d}\xi \\
&\leq N_w \int_0^{T'}\int_{\bR^d}|f(t,x)|^2  w(t)  \mathrm{d}x \mathrm{d}t.
\end{align*}
This establishes the desired estimate and concludes the proof.
\end{proof}

\begin{rem}
									\label{phi zero lemma}
In the proof of the preceding lemma, it might appear that strict positivity is required to establish the bound:
$$
\int^\infty_{0}  \phi(\xi) \exp\left( -s \phi(\xi) \right)  \mathrm{d}s \leq 1.
$$
If $\phi(\xi) = 0$, the integrand vanishes identically, so the bound holds trivially at that frequency.
 More precisely, by examining the terms prior to applying Plancherel's theorem, we obtain the following inequality for this $\xi$:
$$
\int_0^{T'} \left|\int_0^t \exp\left( \int_s^t\psi(r,\xi) \mathrm{d}r \right) \phi(\xi) \cF[f(s,\cdot)](\xi) \mathrm{d}s\right|^2  w(t) \mathrm{d}t \leq N_w \left(\int_{0}^{T'}  |\cF[f(t,\cdot)](\xi)|^2 w(t) \mathrm{d}t \right).
$$
This directly yields \eqref{maximal regularity estimate}.
 In other words, the maximal regularity estimate \eqref{maximal regularity estimate} remains valid even if $\phi$ is degenerate. 
Therefore, we could relax the condition in Assumption \ref{main as 5} by dropping the strict positivity requirement, permitting $\phi(\xi) = 0$ (and consequently $\psi(t,\xi) = 0$) for $\xi \in \bR^d$.
\end{rem}

\mysection{Existence and uniqueness of solutions}
											\label{ex uni section}

The solution concepts introduced in this paper—strong, spatial Fourier weak, and spectral-limit solutions—are closely interconnected under suitable conditions. 
Consequently, well-posedness could theoretically be established using any of these frameworks. 
We intentionally choose the spatial Fourier weak solution as our primary setup; this strategic choice enables us to build upon our prior results, significantly simplifying and shortening the proofs that follow.

We first establish the uniqueness principle needed below directly in the present
$\mathrm{L}_2$ framework.

\begin{thm}
										\label{uniqueness theorem}
Let $f \in \mathrm{L}_{2,loc}\left( (0,T) \times \bR^d \right)$. 
Assume that 
\begin{equation}
										\label{20260811 01}
\begin{aligned}
&\int_0^t \int_{B_R} |\psi(s,\xi)|\mathrm{d}\xi \mathrm{d}s < \infty
\quad \text{for all } t \in (0,T) \text{ and } R > 0,\\
&\psi(s,\cdot)\in \mathrm{L}_{2,loc}(\bR^d)
\quad \text{for a.e. }s\in(0,T).
\end{aligned}
\end{equation}
Then there is at most one spatial Fourier weak solution $u$ to \eqref{ab eqn} in the space
$$
 \mathrm{L}_{2,loc}\left( (0,T) \times \bR^d \right) \cap \mathrm{H}_{1,2,loc}^{\psi}\left( (0,T) \times \bR^d\right),
$$
where the definition of $\mathrm{H}_{1,2,loc}^{\psi}\left( (0,T) \times \bR^d\right)$ is given in \eqref{psi space}.
\end{thm}
\begin{proof}
Let $u_1$ and $u_2$ be two such solutions and set $v=u_1-u_2$.
By the definition of $\mathrm{H}_{1,2,loc}^{\psi}$,
\[
G(s):=\cF^{-1}\!\left[\psi(s,\cdot)\cF[v(s,\cdot)]\right]
\in \mathrm{L}_1\big((0,t);\mathrm{L}_2(\bR^d)\big)
\]
for every $t<T$.  Subtracting the two weak formulations gives
\[
\left(v(t)-\int_0^tG(s)\,\mathrm{d}s,\varphi\right)_{\mathrm L_2}=0
\]
for every $\varphi\in\cF^{-1}\mathrm C_c^\infty(\bR^d)$ and almost every
$t$.  Choose a countable $\mathrm L_2$-dense subset of this test-function
space and remove the union of the corresponding null sets.  It follows that
\[
v(t)=\int_0^tG(s)\,\mathrm{d}s\qquad\text{in }\mathrm L_2(\bR^d)
\]
for almost every $t$.  After taking the Fourier transform and using Fubini's
theorem on each ball $B_R$, we obtain, for almost every $\xi$,
\[
\widehat v(t,\xi)=\int_0^t\psi(s,\xi)\widehat v(s,\xi)\,\mathrm{d}s
\quad\text{for a.e. }t.
\]
The first part of \eqref{20260811 01} implies that
$\psi(\cdot,\xi)\in\mathrm L_1(0,t)$ for almost every $\xi$.  Hence the
right-hand side has an absolutely continuous representative $V(\cdot,\xi)$
satisfying
\[
V'(t,\xi)=\psi(t,\xi)V(t,\xi)\quad\text{a.e.},
\qquad V(0,\xi)=0.
\]
Multiplication by the integrating factor
$\exp(-\int_0^t\psi(r,\xi)\,\mathrm{d}r)$ shows that $V(t,\xi)=0$.
Thus $v=0$ almost everywhere.
\end{proof}

We next give a direct existence and uniqueness result in the precise
Fourier-side product class used in its proof.

\begin{thm}
									\label{weak well}
Let $f \in \mathrm{L}_{2,loc}\left( (0,T) \times \bR^d \right)$. Suppose
Assumption \ref{main as} and condition \eqref{20260811 01} hold.
Additionally, assume that for all $R \in (0,\infty)$ and $t \in (0,T)$, the following integrability condition is satisfied:
\begin{align}
										\label{20260324 01}
\int_{B_R} \int_0^t |\psi(\rho,\xi)| \int_0^\rho |\cF[f(s,\cdot)](\xi)| \mathrm{d}s \mathrm{d}\rho \mathrm{d}\xi < \infty \quad \forall R \in (0,\infty)~\text{and}~ \forall t \in (0,T).
\end{align}
Then there exists a spatial Fourier weak solution $u$ to \eqref{ab eqn} in the space
$$
\mathrm{L}_{\infty,2, loc}\left( (0,T) \times \bR^d, \frac{1}{t}\mathrm{d}t \mathrm{d}x \right) \cap \mathrm{L}_{2, loc}\left( (0,T) \times \bR^d, \frac{1}{t^2}\mathrm{d}t\mathrm{d}x \right).
$$
It is unique among the solutions in this space that satisfy
\begin{align}
										\label{20260819 01}
\int_0^t\int_{B_R}|\psi(s,\xi)|\,|\cF[u(s,\cdot)](\xi)|
\,\mathrm{d}\xi\mathrm{d}s<\infty
\quad\text{for all }t<T\text{ and }R>0.
\end{align}
Furthermore, $u$ is given by \eqref{u defn} and satisfies
\eqref{20260321 01-2} and \eqref{20260321 02}.
\end{thm}
\begin{proof}
Condition \eqref{20260811 01} implies Assumption \ref{main as 2} by
Fubini's theorem.  Define $u$ by \eqref{u defn}.  Since
$\Re\psi\leq0$, Lemma \ref{l2 bounded lemma 1} gives
\eqref{20260321 01-2} and \eqref{20260321 02}, and therefore places $u$ in
the two asserted spaces.

For almost every $\xi$, both $\psi(\cdot,\xi)$ and the function
$s\mapsto\cF[f(s,\cdot)](\xi)$ are locally integrable in time.  The scalar
variation-of-constants formula therefore yields
\begin{align}
										\label{20260819 02}
\cF[u(t,\cdot)](\xi)
=\int_0^t\psi(s,\xi)\cF[u(s,\cdot)](\xi)\,\mathrm{d}s
+\int_0^t\cF[f(s,\cdot)](\xi)\,\mathrm{d}s
\end{align}
for almost every $(t,\xi)$.  Moreover,
\[
|\cF[u(\rho,\cdot)](\xi)|
\leq\int_0^\rho|\cF[f(s,\cdot)](\xi)|\,\mathrm{d}s,
\]
so \eqref{20260324 01} proves \eqref{20260819 01}.  If
$\varphi\in\cF^{-1}\mathrm C_c^\infty(\bR^d)$ and
$\operatorname{supp}\widehat\varphi\subset B_R$, the second part of
\eqref{20260811 01} ensures that
$\overline\psi(s,-\mathrm i\nabla)\varphi\in\mathrm L_2$ for almost every
$s$, while \eqref{20260819 01} justifies Fubini's theorem in
\eqref{20260819 02}.  Multiplying that identity by
$\overline{\widehat\varphi}$, integrating in $\xi$, and applying
Plancherel's theorem gives \eqref{weak formulation}.  Thus $u$ is a spatial
Fourier weak solution.

For uniqueness, let $u_1,u_2$ be two solutions in the asserted class that
satisfy \eqref{20260819 01}, and put $v=u_1-u_2$.  Subtracting their weak
formulations and using a countable dense family in
$\mathrm C_c^\infty(B_R)$ gives, simultaneously outside one null set,
\[
\widehat v(t,\xi)=\int_0^t\psi(s,\xi)\widehat v(s,\xi)\,\mathrm{d}s
\quad\text{for a.e. }(t,\xi).
\]
Here \eqref{20260819 01} justifies the local Fourier-side integrals and the
passage from the test identities to the displayed equality.  The same
integrating-factor argument as in Theorem \ref{uniqueness theorem}, using
the first part of \eqref{20260811 01}, yields $v=0$.
\end{proof}

\begin{rem}
										\label{20260401 rmk 1}

Assuming that Assumption \ref{main as 3} is satisfied and $f \in \mathrm{L}_{2,loc}\left( (0,T) \times \bR^d, w(t)\mathrm{d}t\mathrm{d}x \right)$, condition \eqref{20260324 01} holds automatically. 
While straightforward to verify, the details are deferred to the proof of Theorem \ref{main thm 2} in Section \ref{pf main thm}. 
This sufficient condition plays a pivotal role in establishing Theorem \ref{main thm 2}

To optimize this framework further, we introduce a relevant corollary immediately following this remark.
It is worth noting that condition \eqref{20260324 01} holds intrinsic value beyond merely satisfying the prerequisites of our main theorem. 
Specifically, it exposes a critical structural trade-off: the symbol $\psi$ is allowed to exhibit highly singular behavior, provided that the inhomogeneous data $f$ is sufficiently well-behaved to control those singularities.
\end{rem}

\begin{corollary}
									\label{weak well weight}

Let $w$ be a weight satisfying Assumption \ref{weight as}, and let $f \in \mathrm{L}_{2,loc}\left( (0,T) \times \bR^d, w(t)\mathrm{d}t\mathrm{d}x \right)$. 
Suppose that Assumption \ref{main as} holds. 
Additionally, assume \eqref{20260811 01} and \eqref{20260324 01} are true.
Then \eqref{ab eqn} admits a spatial Fourier weak solution $u$ in the space
$$
\mathrm{L}_{\infty,2, loc}\left( (0,T) \times \bR^d, \frac{w(t)}{t} \mathrm{d}t \mathrm{d}x \right) \cap \mathrm{L}_{2, loc}\left( (0,T) \times \bR^d, \frac{w(t)}{t^2}\mathrm{d}t \mathrm{d}x \right).
$$
It is unique among the solutions in this space satisfying
\eqref{20260819 01}. Furthermore, $u$ is given by \eqref{u defn} and
satisfies \eqref{20260321 03-2} and \eqref{20260321 04}.
\end{corollary}
\begin{proof}
Recalling Remark \ref{extra weight rem}, we have the inclusion:
$$
\mathrm{L}_{2,loc}\left( (0,T) \times \bR^d, w(t)\mathrm{d}t\mathrm{d}x \right) \subset \mathrm{L}_{2,loc}\left( (0,T) \times \bR^d \right).
$$
Therefore, by Theorem \ref{weak well}, there exists a spatial Fourier weak solution $u$ to \eqref{ab eqn}, unique among those satisfying \eqref{20260819 01}, in
$$
\mathrm{L}_{\infty,2, loc}\left( (0,T) \times \bR^d, \frac{1}{t}\mathrm{d}t \mathrm{d}x \right) \cap \mathrm{L}_{2,loc}\left( (0,T) \times \bR^d, \frac{1}{t^2}\mathrm{d}t\mathrm{d}x \right),$$
such that \eqref{20260321 01-2} and \eqref{20260321 02} are satisfied. 
Moreover, Corollary \ref{l2 bounded corollary 1} ensures that \eqref{20260321 03-2} and \eqref{20260321 04} hold. 
Finally, it is evident that \eqref{20260321 03-2} and \eqref{20260321 04} imply
$$
u \in \mathrm{L}_{\infty,2, loc}\left( (0,T) \times \bR^d, \frac{w(t)}{t} \mathrm{d}t \mathrm{d}x \right) \cap \mathrm{L}_{2, loc}\left( (0,T) \times \bR^d, \frac{w(t)}{t^2}\mathrm{d}t \mathrm{d}x \right),
$$
which completes the proof.
\end{proof}

Recall that establishing our well-posedness theory relied on selecting the spatial Fourier weak solution as our fundamental starting point, a choice that seamlessly bridged our earlier findings. 
The relation between strong and spatial Fourier weak solutions is summarized in Lemma \ref{inculsuion solutions lemma}; convergence to a spectral-limit solution is proved separately in Theorem \ref{main thm}.

The substantial mathematical challenge, rather, lies in establishing higher regularity. 
Transforming a weak solution into a strong solution is highly non-trivial. 
Therefore, the main focus of this section is to resolve this issue by determining the precise conditions on the symbol that ensure this regularity enhancement, effectively converting the spatial Fourier weak solution into a strong one.

Before proceeding, we introduce two function classes defined via the Fourier transform. 
We say that
$$
f \in \mathrm{L}_{2,loc}\left( (0,T) \times \bR^d, w(t)\mathrm{d}t\mathrm{d}x \right) 
\cap \cF^{-1}\mathrm{L}_{1,loc}\left( (0,T) \times \bR^d \right)
$$
provided $f \in \mathrm{L}_{2,loc}\left( (0,T) \times \bR^d, w(t)\mathrm{d}t\mathrm{d}x \right)$ and
$$
\int_0^t\int_{\bR^d}|\cF[f(s,\cdot)](\xi)| \mathrm{d}\xi \mathrm{d}s < \infty \quad \forall t \in (0,T).
$$
Similarly, in a formal sense, we write
$$
u \in \cF^{-1}\mathrm{H}_{1,loc}^{\psi}\left( (0,T) \times \bR^d\right)
$$
whenever
$$
\int_0^t\int_{\bR^d} |\psi(s,\xi)| |\cF[u(s,\cdot)](\xi)| \mathrm{d}\xi \mathrm{d}s < \infty \quad \forall t \in (0,T).
$$

\begin{lem}
										\label{strong solution lem}
Let $f \in \mathrm{L}_{2,loc}\left( (0,T) \times \bR^d, w(t)\mathrm{d}t\mathrm{d}x \right) 
\cap \cF^{-1}\mathrm{L}_{1,loc}\left( (0,T) \times \bR^d\right)$, 
where $w$ is a weight satisfying Assumption \ref{weight as}. 
Assume in addition that \eqref{20260811 01} holds.
Suppose $u$ is a spatial Fourier weak solution to \eqref{ab eqn} that belongs to the intersection of the following four function spaces:
$$
\begin{aligned}
&\mathrm{L}_{\infty,2, loc}\left( (0,T) \times \bR^d ,\frac{w(t)}{t}\mathrm{d}t \mathrm{d}x\right),\\
&\mathrm{L}_{2, loc}\left( (0,T) \times \bR^d,\frac{w(t)}{t^2}\mathrm{d}t\mathrm{d}x \right),\\
& \mathrm{H}_{1,2,loc}^{\psi}\left( (0,T) \times \bR^d\right), \\
& \cF^{-1}\mathrm{H}_{1,loc}^{\psi}\left( (0,T) \times \bR^d\right).
\end{aligned}
$$
Then, $u$ is the unique strong solution to \eqref{ab eqn} within this class.

\end{lem}

\begin{proof}
We first show that the spatial Fourier weak solution $u$ is a strong solution.
For any test function $\varphi \in \cF^{-1}\mathrm{C}_c^\infty(\bR^d)$, the weak formulation yields:
\begin{align} 
									\label{20230609 01} 
\int_{\bR^d} u(t,x) \overline{\varphi(x)} \mathrm{d}x =  \int_0^t \int_{\bR^d}u(s,x) \overline{ \overline{\psi}(s,-\mathrm{i}\nabla) \varphi(x)} \mathrm{d}x \mathrm{d}s + \int_0^t \int_{\bR^d} f(s,x)  \overline{\varphi(x)} \mathrm{d}x \mathrm{d}s 
\end{align}
for a.e. $t\in (0,T)$.
Given that the weight $w$ satisfies Assumption \ref{weight as} and
\[
u \in \mathrm{L}_{2, loc}\left( (0,T) \times \bR^d, t^{-2}w(t) \mathrm{d}t \mathrm{d}x \right),
\]
we obtain the following bound for all $t \in (0,T)$:
\[
\int_0^t \int_{\bR^d} | u(s,x)|^2 \mathrm{d}x \mathrm{d}s \lesssim_{t,N_w} \int_0^t \int_{\bR^d} | u(s,x)|^2 \frac{w(s)}{s^2} \mathrm{d}x \mathrm{d}s < \infty.
\]
In particular, this implies
\[
\int_{\bR^d} | u(t,x)|^2 \mathrm{d}x < \infty \quad \text{for a.e. } t \in (0,T).
\]
Furthermore, if $\cF[\varphi] \in \mathrm{C}_c^\infty(\bR^d)$, we can bound the term involving the symbol as follows:
\[
\begin{aligned} 
\int_{\bR^d} \left| \psi(t,\xi) \overline{\cF[\varphi](\xi)} \right|^2 \mathrm{d}\xi &\leq \|1_{\text{supp } \cF[\varphi]} \psi(t,\cdot)\|^2_{\mathrm{L}_2(\bR^d)} \|  \cF[\varphi] \|^2_{\mathrm{L}_{\infty}(\bR^d)} < \infty 
\end{aligned}
\]
for a.e. $t \in (0,T)$, where $\text{supp } \cF[\varphi]$ denotes the compact support of $\cF[\varphi]$.
Applying Plancherel's theorem to \eqref{20230609 01}, we find that for any $\varphi \in \cF^{-1}\mathrm{C}_c^\infty(\bR^d)$,
\begin{align} 
									\notag
&\int_{\bR^d} \cF[u(t,\cdot)](\xi) \overline{\cF[\varphi](\xi)} \mathrm{d}\xi  \\ 													\label{20230529 01} 
&= \int_0^t \int_{\bR^d} \cF[u(s,\cdot)](\xi) \psi(s,\xi) \overline{\cF[\varphi](\xi)} \mathrm{d}\xi \mathrm{d}s + \int_0^t \int_{\bR^d} \cF[f(s,\cdot)](\xi) \overline{\cF[\varphi](\xi)} \mathrm{d}\xi \mathrm{d}s 
\end{align}
for a.e. $t\in (0,T)$.

We now introduce a Sobolev mollifier. 
Let $\chi \in \cF^{-1}\mathrm{C}_c^\infty(\bR^d)$ be chosen such that its Fourier transform $\cF[\chi]$ is non-negative and symmetric, satisfying the normalization condition:
$$
\int_{\bR^d} \chi(x)\mathrm{d}x=(2\pi)^{d/2}.
$$
For a scaling parameter $\varepsilon \in (0,1)$, we define the scaled mollifier as
$$
\chi^\varepsilon(x) := \frac{1}{\varepsilon^d}\chi\left( \frac{x}{\varepsilon}\right).
$$
Fixing $x \in \bR^d$ and substituting $\chi^\varepsilon(x-\cdot)$ for the test function $\varphi$ in \eqref{20230529 01}, we obtain the following identity for almost every $t \in (0,T)$,
$$
\begin{aligned} 
&\int_{\bR^d} \cF[u(t,\cdot)](\xi) \overline{\cF[\chi^\varepsilon(x-\cdot)](\xi)}\mathrm{d}\xi \\ 
&=\int_0^t \int_{\bR^d} \cF[u(s,\cdot)](\xi) \psi(s,\xi) \overline{\cF[\chi^\varepsilon(x-\cdot)](\xi)} \mathrm{d}\xi \mathrm{d}s +\int_0^t \int_{\bR^d} \cF[f(s,\cdot)](\xi) \overline{\cF[\chi^\varepsilon(x-\cdot)](\xi)} \mathrm{d}\xi \mathrm{d}s.
\end{aligned}
$$
Observe that $\cF[u(t,\cdot)] \cF[\chi^\varepsilon] \in L_1(\bR^d)$ for almost every $t \in (0,T)$, since $\cF[\chi^\varepsilon]$ is a bounded function with compact support.

Consequently, by applying Plancherel's theorem, Fubini's theorem, and a straightforward change of variables, we deduce that for any $x \in \bR^d$,
\begin{align*}
&(2\pi)^{d/2}\cF^{-1}\left[ \cF[u(t,\cdot)] \cF[\chi^\varepsilon] \right](x)  \\
									\notag
&=\int_{\bR^d} e^{i  x \cdot \xi} \cF[u(t,\cdot)](\xi) \cF[\chi^\varepsilon](\xi) \mathrm{d}\xi \\
									\notag
&=\int_{\bR^d} \cF[u(t,\cdot)](\xi)  \overline{e^{-i  x \cdot \xi}\cF[\chi^\varepsilon](-\xi)} \mathrm{d}\xi \\
									\notag
&=\int_{\bR^d} \cF[u(t,\cdot)](\xi) \overline{\cF[\chi^\varepsilon(x-\cdot)](\xi)} \mathrm{d}\xi \\
									\notag
&= \int_0^t \int_{\bR^d} \cF[u(s,\cdot)](\xi) \psi(s,\xi) \overline{ \cF[\chi^\varepsilon(x-\cdot)](\xi)} \mathrm{d}\xi\mathrm{d}s 
+\int_0^t \int_{\bR^d} \cF[f(s,\cdot)](\xi) \overline{\cF[\chi^\varepsilon(x-\cdot)](\xi)} \mathrm{d}\xi \mathrm{d}s \\
									\notag
&= \int_0^t \int_{\bR^d} e^{ix \cdot \xi}\cF[u(s,\cdot)](\xi) \psi(s,\xi)  \cF[\chi^\varepsilon](\xi) \mathrm{d}\xi\mathrm{d}s 
+\int_0^t \int_{\bR^d} e^{ix \cdot \xi}\cF[f(s,\cdot)](\xi)  \cF[\chi^\varepsilon](\xi) \mathrm{d}\xi\mathrm{d}s \\
									\notag
&=(2\pi)^{d/2} \cF_{\xi}^{-1}\left[ \int_0^t \psi(s,\xi) \cF[u(s,\cdot)](\xi)  \cF[\chi^\varepsilon](\xi) \mathrm{d}s \right](x)
+(2\pi)^{d/2} \cF_{\xi}^{-1}\left[ \int_0^t\cF[f(s,\cdot)](\xi)  \cF[\chi^\varepsilon](\xi) \mathrm{d}s \right](x)
\end{align*}
for almost every $t \in (0,T)$. 
The application of Fubini's theorem in this step is rigorously justified by the Cauchy--Bunyakovsky--Schwarz inequality, combined with the integrability assumptions placed on $u$ and $f$.

Furthermore, all terms in the preceding estimates are continuous with respect to $x$, owing to the sufficient regularity imposed on $u$ and $f$.
Consequently, this equality holds for almost every $t \in (0,T)$ and for all $x \in \bR^d$.

Applying the spatial Fourier transform to the first and final expressions in the preceding chain of equalities, we obtain the following equation for almost every $\xi \in \bR^d$:
\begin{align}
								\label{20230215 50}
\cF[u(t,\cdot)](\xi) \cF[\chi^\varepsilon](\xi)
= \int_0^t \psi(s,\xi) \cF[u(s,\cdot)](\xi)  \cF[\chi^\varepsilon](\xi) \mathrm{d}s 
+\int_0^t  \cF[f(s,\cdot)](\xi)  \cF[\chi^\varepsilon](\xi) \mathrm{d}s 
\end{align}
for a.e. $t \in (0,T)$.
Next, we observe the pointwise behavior of the mollifier's Fourier transform as $\varepsilon \downarrow 0$:
\begin{align*}
\lim_{\varepsilon \downarrow 0}\cF[\chi^\varepsilon](\xi)
=\lim_{\varepsilon \downarrow 0}\cF[\chi](\varepsilon \xi)= \cF[\chi](0)= \frac{1}{(2\pi)^{d/2}} \int_{\bR^d}\chi(x)\mathrm{d}x=1 \quad \text{a.e.}~ \xi \in \bR^d.
\end{align*}
Therefore, by setting $\varepsilon=1/n$ and passing to the limit as $n \to \infty$ in \eqref{20230215 50}, we arrive at the following identity for almost every $t \in (0,T)$:
\begin{align*}
\cF[u(t,\cdot)](\xi)
=\int_0^t \psi(s,\xi) \cF[u(s,\cdot)](\xi) \mathrm{d}s 
+\int_0^t  \cF[f(s,\cdot)](\xi) \mathrm{d}s 
\quad \text{a.e.}~ \xi \in \bR^d.
\end{align*}

Finally, we recall that the inverse Fourier transform is a continuous operator, allowing us to utilize Hille's theorem (\textit{cf.} \cite[Corollary V.5.2]{Yosida}) to interchange the inverse Fourier transform with the temporal integrals. 
The hypotheses required for Hille's theorem are clearly satisfied; based on our assumptions regarding $u$ and $f$, we have the following bound for any $t \in (0,T)$:
\begin{align*}
\int_0^t \left( \int_{\bR^d} \left|\psi(s,\xi) \cF[u(s,\cdot)](\xi) \right|^2 \mathrm{d}\xi \right)^{1/2} \mathrm{d}s 
+ \int_0^t \left( \int_{\bR^d} \left|\cF[f(s,\cdot)](\xi)\right|^2 \mathrm{d}\xi \right)^{1/2} \mathrm{d}s < \infty.
\end{align*}
Therefore, by applying the inverse Fourier transform to both sides and invoking Hille's theorem, we definitively conclude that $u$ is a strong solution to \eqref{ab eqn}.

It remains to prove uniqueness in the stated class.  By Lemma
\ref{strong unique lem} below, every strong solution in this class is a
spatial Fourier weak solution.  Condition \eqref{20260811 01} and Theorem
\ref{uniqueness theorem} then imply that any two such strong solutions
coincide.

\end{proof}

By combining Lemma \ref{l2 bounded lemma 1} and the preceding lemma, we establish the well-posedness of a strong solution to \eqref{ab eqn} for smooth data $f$, without relying on \textit{a priori} estimates.

\begin{thm}
										\label{strong solution thm}
Let the function $f$ belong to the weighted space 
$$
\mathrm{L}_{2,loc}\left( (0,T) \times \bR^d, w(t)\mathrm{d}t \mathrm{d}x \right)\cap \cF^{-1}\mathrm{L}_{1,loc}\left( (0,T) \times \bR^d\right), 
$$
and let the weight function $w$ satisfy the conditions of Assumption \ref{weight as}. Assume that Assumptions \ref{main as} and \ref{main as 2} and condition \eqref{20260811 01} hold true. In addition, suppose that the following integral condition is finite for every $t \in (0,T)$:
\begin{align}
									\label{20260401 01}
\int_0^t  \left(\int_{\bR^d}  \left( |\psi(\rho,\xi)| \int_0^\rho |\cF[f(s,\cdot)](\xi) \mathrm{d}s\right)^2  \mathrm{d}\xi \right)^{1/2} \mathrm{d}\rho
+\int_0^t  \int_{\bR^d}   |\psi(\rho,\xi)| \int_0^\rho |\cF[f(s,\cdot)](\xi) \mathrm{d}s  \mathrm{d}\xi \mathrm{d}\rho
<\infty.
\end{align}
Under these conditions, there is exactly one strong solution $u$ to equation \eqref{ab eqn} that belongs to the intersection of the following local spaces:
$$
\begin{aligned}
&\mathrm{L}_{\infty,2, loc}\left( (0,T) \times \bR^d ,\frac{w(t)}{t}\mathrm{d}t \mathrm{d}x\right), \mathrm{L}_{2, loc}\left( (0,T) \times \bR^d,\frac{w(t)}{t^2}\mathrm{d}t\mathrm{d}x \right), \mathrm{H}_{1,2,loc}^{\psi}\left( (0,T) \times \bR^d\right)
, 
\end{aligned}
$$
and $\cF^{-1}\mathrm{H}_{1,loc}^{\psi}\left( (0,T) \times \bR^d\right)$.
Furthermore, this solution $u$ satisfies the inequalities \eqref{20260321 03-2} and \eqref{20260321 04}, and it can be explicitly represented by the formula \eqref{u defn}.

\end{thm}

\begin{proof}
First, condition \eqref{20260401 01} implies the local integrability condition \eqref{20260324 01}. 
Therefore, Corollary \ref{weak well weight} yields a spatial Fourier weak solution $u$, given by \eqref{u defn}, in the two weighted spaces stated in the theorem and satisfying \eqref{20260321 03-2} and \eqref{20260321 04}; it is unique among such solutions satisfying \eqref{20260819 01}.

It remains to verify the operator-integrability condition required by Lemma \ref{strong solution lem}. Since the real part of $\psi$ is non-positive, the representation \eqref{u defn} implies
\[
|\cF[u(\rho,\cdot)](\xi)|
\leq \int_0^\rho |\cF[f(s,\cdot)](\xi)|\,\mathrm{d}s.
\]
Consequently, \eqref{20260401 01} yields
\begin{align*}
\int_0^t\left(\int_{\bR^d}|\psi(\rho,\xi)\cF[u(\rho,\cdot)](\xi)|^2\mathrm{d}\xi\right)^{1/2}\mathrm{d}\rho
+\int_0^t\int_{\bR^d}|\psi(\rho,\xi)\cF[u(\rho,\cdot)](\xi)|\mathrm{d}\xi\mathrm{d}\rho
<\infty
\end{align*}
for every $t\in(0,T)$. Thus
\[
u\in\mathrm{H}_{1,2,loc}^{\psi}\left( (0,T) \times \bR^d\right) \cap \cF^{-1}\mathrm{H}_{1,loc}^{\psi}\left( (0,T) \times \bR^d\right).
\]
Lemma \ref{strong solution lem} now upgrades $u$ to the unique strong solution in the asserted class.
\end{proof}

One might assume, as noted in Remark \ref{20260401 rmk 1}, that condition \eqref{20260401 01} alone could control the singularity of the symbol by imposing stricter requirements on the inhomogeneous data $f$. 
However, a significant discrepancy exists between \eqref{20260401 01} and \eqref{20260324 01}, primarily because the condition in \eqref{20260324 01} is purely local.
As a result, \eqref{20260401 01} demands that $f$ exhibit strong decay at infinity in the frequency domain. 
This dictates that $f$ must possess high spatial regularity in order to apply the previous theorem. 
Consequently, this theorem falls short when attempting to establish a strong solution for irregular data $f$. 
To guarantee a strong solution to \eqref{ab eqn} for more general data $f$, we must instead employ an approximation argument that relies on the continuity of the main operator; this approach is detailed in the final section.
Regarding uniqueness, our strategy for proving the uniqueness of the strong solution to \eqref{ab eqn} relies on leveraging the uniqueness of its spatial Fourier weak solution. 
This requires identifying the precise mathematical criteria under which a strong solution also qualifies as a spatial Fourier weak solution. 
While verifying this equivalence is trivial for well-behaved operators like the Laplacian, the severe singularities present in our main operator make this a highly non-trivial problem.
Fortunately, this equivalence is valid under the hypotheses of Lemma \ref{strong solution lem}, which outlines the conditions required for a spatial Fourier weak solution to upgrade to a strong solution. 
The majority of these conditions fundamentally deal with justifying the interchange of mathematical operations, such as swapping integration with the Fourier transform.

\begin{lem}
										\label{strong unique lem}
Let the source term $f$ belong to the space $\mathrm{L}_{2,loc}\left( (0,T) \times \bR^d, w(t)\mathrm{d}t\mathrm{d}x \right)$, where the weight function $w$ satisfies Assumption \ref{weight as}. 
Assume in addition that \eqref{20260811 01} holds.
Assume that $u$ is a strong solution to equation \eqref{ab eqn} and resides in the intersection of the following three local spaces:
$$
\begin{aligned}
&\mathrm{L}_{\infty,2, loc}\left( (0,T) \times \bR^d ,\frac{w(t)}{t}\mathrm{d}t \mathrm{d}x\right),\\
&\mathrm{L}_{2, loc}\left( (0,T) \times \bR^d,\frac{w(t)}{t^2}\mathrm{d}t\mathrm{d}x \right),\\
&\mathrm{H}_{1,2,loc}^{\psi}\left( (0,T) \times \bR^d\right).
\end{aligned}
$$
Under these conditions, $u$ is also the unique spatial Fourier weak solution to \eqref{ab eqn} within this function class. 
Consequently, any strong solution is unique within this intersection space.
\end{lem}
\begin{proof}
Fix $t<T$.  Choose $\tau\in(t,T)$ such that $0<w(\tau)<\infty$ and the
quasi-decreasing inequality in Assumption \ref{weight as} holds with terminal
time $\tau$.  Then $w(s)\geq w(\tau)/N_w$ for almost every $s<t$.
Consequently,
\[
\int_0^t\|f(s,\cdot)\|_{\mathrm L_2}\,\mathrm{d}s
\leq t^{1/2}\left(\frac{N_w}{w(\tau)}
\int_0^t\|f(s,\cdot)\|_{\mathrm L_2}^2w(s)\,\mathrm{d}s\right)^{1/2}<\infty.
\]
The assumption $u\in\mathrm H_{1,2,loc}^{\psi}$ similarly gives
\[
\int_0^t\|\psi(s,-\mathrm i\nabla)u(s,\cdot)\|_{\mathrm L_2}
\,\mathrm{d}s<\infty.
\]
Thus the strong identity is an equality of $\mathrm L_2$-valued Bochner
integrals for almost every $t$.

Let $\varphi\in\cF^{-1}\mathrm C_c^\infty(\bR^d)$.  Taking its
$\mathrm L_2$ inner product with the strong identity and using the two
preceding estimates gives
\[
(u(t),\varphi)_{\mathrm L_2}
=\int_0^t(\psi(s,-\mathrm i\nabla)u(s),\varphi)_{\mathrm L_2}\,\mathrm{d}s
+\int_0^t(f(s),\varphi)_{\mathrm L_2}\,\mathrm{d}s.
\]
The second part of \eqref{20260811 01} ensures that
$\overline\psi(s,-\mathrm i\nabla)\varphi\in\mathrm L_2$ for almost every
$s$.  Plancherel's theorem then gives
\[
(\psi(s,-\mathrm i\nabla)u(s),\varphi)_{\mathrm L_2}
=(u(s),\overline\psi(s,-\mathrm i\nabla)\varphi)_{\mathrm L_2}.
\]
Substitution proves the weak formulation \eqref{weak formulation}; hence
$u$ is a spatial Fourier weak solution.  Theorem \ref{uniqueness theorem}
now gives uniqueness in the stated class.  Applying the same argument to any
other strong solution in that class proves the final assertion as well.

\end{proof}

\mysection{A dense subclass and continuity of operators}

In this section, we introduce a suitably regular dense subset of $\mathrm{L}_{2,loc}\left( (0,T) \times \bR^d, w(t)\mathrm{d}t\mathrm{d}x \right)$, which plays a pivotal role in constructing a strong solution to \eqref{ab eqn}. 
Our approach relies on a standard approximation argument prevalent in the theory of partial differential equations. 
Specifically, we approximate the general inhomogeneous data $f$ using a sequence of well-behaved data $f_n$ chosen from this dense subset. 
After establishing a strong solution $u_n$ to \eqref{ab eqn} for each regularized source term $f_n$, we demonstrate that the sequence $\{u_n\}$ converges to the desired strong solution corresponding to the original, general data $f$.

\begin{thm}
Let us define the set $\cA$ as the collection of finite sums of separable functions over $(0,T) \times \bR^d$, specifically,
\begin{align}
										\label{20230220 30}
\cA:=\left\{ \sum_{i=1}^n  1_{(a_i,b_i)}(t) f_i(x) : n \in \bN,~ 0<a_i<b_i<T,~ f_i \in \cF^{-1}\mathrm{C}_c^\infty(\bR^d) \right\}.
\end{align}
Then the function class $\cA$ is dense in the mixed-norm space $\mathrm{L}_{p,q}\left( (0,T) \times \bR^d \right)$ for all exponents $p, q \in (1,\infty)$.
Additionally, $\cA$ is dense in  $\mathrm{L}_{p,q,loc}\left( (0,T) \times \bR^d \right)$.
\end{thm}

\begin{proof}

Observe that $\mathrm{L}_{p,q}\left( (0,T) \times \bR^d \right)$ is dense in $\mathrm{L}_{p,q,loc}\left( (0,T) \times \bR^d \right)$, as the topology of our local spaces is generated by the family of seminorms on $\mathrm{L}_{p,q}\left( (0,T') \times \bR^d \right)$ for all $T' \in (0,T)$. 
Consequently, the second assertion follows immediately, allowing us to focus solely on proving the main statement.

We proceed by contradiction. Let $p, q \in (1,\infty)$ and assume that the set $\cA$ is not dense in
$$
\mathrm{L}_{p,q}\left( (0,T) \times \bR^d \right).
$$
By standard duality arguments, there must exist a non-zero function 
$$
v \in \mathrm{L}_{p',q'}\left( (0,T) \times \bR^d \right)
$$
—where $p'$ and $q'$ are the Hölder conjugates of $p$ and $q$, respectively—that annihilates $\cA$. Specifically, this means
$$
\sum_{i=1}^n \int_0^T \int_{\bR^d} v(t,x) 1_{(a_i,b_i)}(t) f_i(x) \mathrm{d}x \mathrm{d}t = 0
$$
for all $n \in \bN$, intervals $0 < a_i < b_i < T$, and spatial functions $f_i \in \cF^{-1}\mathrm{C}_c^\infty(\bR^d)$.
It is a known fact that the space $\cF^{-1}\mathrm{C}_c^\infty(\bR^d)$ is dense in $L_{q}(\bR^d)$. Furthermore, any open set $\cU \subset (0,T)$ can be constructed as a countable union of disjoint open intervals $(a_i, b_i)$. 
Because the integral equals zero for all such temporal intervals and over a dense set of spatial test functions, it necessarily follows that $v(t,x) = 0$ for almost every $(t,x) \in (0,T) \times \bR^d$.
This forces $v$ to be the zero function, which directly contradicts our initial assumption that $v$ is a non-zero element of $\mathrm{L}_{p',q'}\left( (0,T) \times \bR^d \right)$. 
Therefore, the set $\cA$ must be dense, completing the proof.
\end{proof}

Although the Laplacian operator $\Delta$ is unbounded on $L_2(\bR^d)$, it is defined on a Sobolev space that is dense within $L_2(\bR^d)$. We aim to identify a similar dense subset in which our strong solution resides. 
Because we are dealing with an evolution equation that incorporates a temporal weight, the necessary formulation is inherently more complex. 
Nevertheless, proving this density remains straightforward because we have already established a highly regular dense subset that naturally embeds into many relevant function spaces. 
The following corollary provides a concrete example demonstrating this property.

\begin{corollary}
									\label{dense cor}
Let $w$ denote a weight function satisfying Assumption \ref{weight as}, and let $\phi$ be a complex-valued function defined almost everywhere on $\bR^d$. 
Under these conditions, the scaled set
$$
\frac{1}{\sqrt{w}} \cA := \left\{ \frac{1}{\sqrt{w(t)}} \sum_{i=1}^n 1_{(a_i,b_i)}(t) f_i(x) : n \in \bN, \; 0 < a_i < b_i < T, \; f_i \in \cF^{-1}\mathrm{C}_c^\infty(\bR^d) \right\}
$$
is dense in the weighted space $\mathrm{L}_{2,loc}\left( (0,T) \times \bR^d, w(t)\mathrm{d}t\mathrm{d}x \right)$.
Then, as a direct result, assuming \eqref{20260811 01} holds, those $f$ belonging to
 $$
\mathrm{L}_{2,loc}\left( (0,T) \times \bR^d, w(t)\mathrm{d}t \mathrm{d}x \right)\cap \cF^{-1}\mathrm{L}_{1,loc}\left( (0,T) \times \bR^d \right), 
$$ 
that satisfy \eqref{20260401 01} form a dense subspace of $\mathrm{L}_{2,loc}\left( (0,T) \times \bR^d, w(t)\mathrm{d}t\mathrm{d}x \right)$.
\end{corollary}
\begin{proof}
First, observe that if a function $f$ belongs to the weighted space $\mathrm{L}_{2,loc}\left( (0,T) \times \bR^d, w(t)\mathrm{d}t\mathrm{d}x \right)$, then the product $f \sqrt{w}$ naturally resides in the standard, unweighted space $\mathrm{L}_{2,loc}\left( (0,T) \times \bR^d \right)$.
By the preceding theorem, the class $\cA$ is dense in $\mathrm{L}_{2,loc}\left( (0,T) \times \bR^d \right)$. 
Therefore, the function $f \sqrt{w}$ can be approximated by elements from $\cA$. Dividing this approximation by $\sqrt{w}$ implies that the set $\frac{1}{\sqrt{w}} \cA$ is dense in $\mathrm{L}_{2,loc}\left( (0,T) \times \bR^d, w(t)\mathrm{d}t\mathrm{d}x \right)$.

It remains to prove the second assertion. 
Let $\mathscr{D}$ denote the class of functions appearing therein. 
Since $\cA/\sqrt w$ is dense in the weighted space, it suffices to prove that
$$
\frac{1}{\sqrt w}\cA\subset\mathscr D.
$$
Fix
$$
g(s,x)
=
\frac{1}{\sqrt{w(s)}}
\sum_{i=1}^n1_{(a_i,b_i)}(s)f_i(x)
\in\frac{1}{\sqrt w}\cA,
$$
and set
$$
a_*:=\min_{1\leq i\leq n}a_i,
\qquad
b_*:=\max_{1\leq i\leq n}b_i,
\qquad
K:=\bigcup_{i=1}^n\operatorname{supp}\cF[f_i].
$$
By Assumption~\ref{weight as}, $w^{-1/2}$ is essentially bounded on
$(a_*,b_*)$. Since $\cF[f_i]\in C_c^\infty(\bR^d)$, it follows immediately
that
$$
g\in
\mathrm{L}_{2,loc}\left((0,T)\times\bR^d,w(t)\,\mathrm dt\,\mathrm dx\right)
\cap
\cF^{-1}\mathrm{L}_{1,loc}\left((0,T)\times\bR^d\right).
$$

Define
$$
G(\rho,\xi)
:=
\int_0^\rho
\left|\cF[g(s,\cdot)](\xi)\right|\,\mathrm ds.
$$
For some constant $C_g>0$,
$$
G(\rho,\xi)
\leq
C_g1_{[a_*,T)}(\rho)1_K(\xi).
$$
Thus, when $t\leq a_*$, both terms in \eqref{20260401 01} vanish. For
$t>a_*$, \eqref{20260811 01}, the local boundedness of $w$, and the
Cauchy--Schwarz inequality yield
\begin{align*}
\int_{a_*}^t
\left(
    \int_K|\psi(\rho,\xi)|^2\,\mathrm d\xi
\right)^{1/2}
\mathrm d\rho
&\leq
\left(
    \int_{a_*}^t\int_K
    |\psi(\rho,\xi)|^2\frac{\rho^2}{w(\rho)}
    \,\mathrm d\xi\,\mathrm d\rho
\right)^{1/2}
\left(
    \int_{a_*}^t\frac{w(\rho)}{\rho^2}\,\mathrm d\rho
\right)^{1/2}
<\infty.
\end{align*}
Moreover, since $K$ has finite measure,
$$
\int_{a_*}^t\int_K|\psi(\rho,\xi)|\,\mathrm d\xi\,\mathrm d\rho
\leq
|K|^{1/2}
\int_{a_*}^t
\left(
    \int_K|\psi(\rho,\xi)|^2\,\mathrm d\xi
\right)^{1/2}
\mathrm d\rho
<\infty.
$$
Together with the bound on $G$, these estimates imply \eqref{20260401 01}. 
Hence $\cA/\sqrt w\subset\mathscr D$, and the second assertion follows from the density of $\cA/\sqrt w$.
\end{proof}

Next, our objective is to formulate a straightforward criterion that ensures the continuity of the principal operator $\psi(t,-\mathrm{i}\nabla)$ as a mapping from the weighted local $\phi$-potential space $\mathrm{H}^\phi_{2,loc}\left((0,T)\times \bR^d , w(t) \mathrm{d}t\mathrm{d}x\right)$ into $\mathrm{L}_{1,2,loc}\left((0,T)\times \bR^d \right)$. 
Establishing this continuity property is a vital component in proving the existence of a strong solution. 
To consolidate all the fundamental tools required for this construction, we present the following theorem in this section, despite it being largely independent of the regular dense subset established previously. 
Additionally, the subsequent corollary is stated in a somewhat broader context to facilitate its use in future applications.

\begin{thm}
										\label{continuity operator}
Let $w$, $\phi$, and $\psi$ be complex-valued functions defined almost everywhere on $(0,T)$, $\bR^d$, and $(0,T) \times \bR^d$, respectively, and assume that $w(t)\ne 0$ for almost every $t\in(0,T)$. We use the convention $\psi(t,\xi)/\phi(\xi)=0$ when $\psi(t,\xi)=\phi(\xi)=0$, and interpret its modulus as $+\infty$ when $\phi(\xi)=0$ but $\psi(t,\xi)\ne0$. 
Suppose that the following integrability condition holds:
\begin{align}
										\label{20260406 10}
\int_0^t \esssup_{\xi \in \bR^d} \left| \frac{\psi(s,\xi)}{\phi(\xi)} \right|^2 \frac{1}{|w(s)|} \mathrm{d}s < \infty \quad \text{for all }t\in(0,T).
\end{align}
Under this assumption, the mapping $u \mapsto \psi(t,-\mathrm{i}\nabla)u$ is a continuous operator from the weighted space $\mathrm{H}^\phi_{2,loc}\left((0,T)\times \bR^d , |w(t)| \mathrm{d}t\mathrm{d}x\right)$ into $\mathrm{L}_{1,2,loc}\left((0,T)\times \bR^d \right)$. 
More precisely, it satisfies the following quantitative bound:
\begin{align*}
&\int_0^t \left( \int_{\bR^d} \left| \psi(s,-\mathrm{i}\nabla) u(s,x) \right|^2 \mathrm{d}x \right)^{1/2} \mathrm{d}s \\
&\leq \left(\int_0^t \esssup_{\xi \in \bR^d} \left| \frac{\psi(s,\xi)}{\phi(\xi)} \right|^2 \frac{1}{|w(s)|} \mathrm{d}s \right)^{1/2} \left( \int_0^t \int_{\bR^d} \left| \phi(-\mathrm{i}\nabla) u(s,x) \right|^2 \mathrm{d}x |w(s)| \mathrm{d}s \right)^{1/2}.
\end{align*}
Consequently, for any function $u \in \mathrm{H}^\phi_{2,loc}\left((0,T)\times \bR^d , |w(t)| \mathrm{d}t\mathrm{d}x\right)$, the term $\psi(t,-\mathrm{i}\nabla) u(t,x)$ is well-defined almost everywhere on $(0,T) \times \bR^d$. 
This establishes the continuous embedding
$$
\mathrm{H}^\phi_{2,loc}\left((0,T)\times \bR^d , |w(t)| \mathrm{d}t\mathrm{d}x\right) \hookrightarrow \mathrm{H}_{1,2,loc}^{\psi}\left( (0,T) \times \bR^d\right).
$$
\end{thm}
\begin{proof}
This result is a direct consequence of applying Plancherel's theorem followed by the Cauchy-Bunyakovsky-Schwarz inequality. 
By transforming into the frequency domain and estimating the integrals, we deduce
\begin{align*}
 &\int_0^t \left( \int_{\bR^d}  \left| \psi(s,-\mathrm{i}\nabla) u(s,x) \right|^2 \mathrm{d}x  \right)^{1/2} \mathrm{d}s  \\
&= \int_0^t \left( \int_{\bR^d}  \left| \psi(s,\xi) \cF[u(s,\cdot)](\xi) \right|^2 \mathrm{d}\xi  \right)^{1/2} \mathrm{d}s\\
&\leq  \int_0^t  \esssup_{\xi \in \bR^d} \left(\left| \frac{\psi(s,\xi)}{\phi(\xi)} \right|\right) \left( \int_{\bR^d}  \left| \phi(\xi) \cF[u(s,\cdot)](\xi) \right|^2 \mathrm{d}\xi  \right)^{1/2} \mathrm{d}s \\
&\leq  \left(\int_0^t  \esssup_{\xi \in \bR^d} \left(\left| \frac{\psi(s,\xi)}{\phi(\xi)} \right|\right)^2\frac{1}{|w(s)|} \mathrm{d}s \right)^{1/2}
\left( \int_0^t \int_{\bR^d}  \left| \phi(\xi) \cF[u(s,\cdot)](\xi) \right|^2 \mathrm{d}\xi  |w(s)|\mathrm{d}s \right)^{1/2} \\
&=  \left(\int_0^t  \esssup_{\xi \in \bR^d} \left(\left| \frac{\psi(s,\xi)}{\phi(\xi)} \right|\right)^2 \frac{1}{|w(s)|} \mathrm{d}s \right)^{1/2}
\left( \int_0^t \int_{\bR^d}  \left| \phi(-\mathrm{i}\nabla) u(s,x) \right|^2 \mathrm{d}x |w(s)| \mathrm{d}s \right)^{1/2}.
\end{align*}
This completes the proof.
\end{proof}

\begin{rem}
										\label{degenerate operator}
We reiterate that $\phi$ is permitted to degenerate in \eqref{20260406 10}.
 In the event that $\phi(\xi)=0$ and $\psi(t,\xi)=0$, the inequality
$$
|\psi(t,\xi)| \leq  |\phi(\xi)| \esssup_{\eta \in \bR^d} \left| \frac{\psi(t,\eta)}{\phi(\eta)} \right|
$$
remains valid under the convention stated in Theorem \ref{continuity operator}. 

\end{rem}

\mysection{Proofs of main theorems}
									\label{pf main thm}

We begin this section by establishing the inclusion relationships among the various notions of solutions mentioned in Remark \ref{solution inculsion remark}. 
Since most of these results follow either from direct verification or from previously established facts, our primary goal here is simply to summarize them in the subsequent lemma.

\begin{lem}
										\label{inculsuion solutions lemma}
Let $f \in \mathrm{L}_{2,loc}((0,T) \times \bR^d, w(t)\mathrm{d}t\mathrm{d}x)$ and suppose $w$ is a weight satisfying Assumption \ref{weight as}. Assume also that \eqref{20260811 01} holds.
If $u$ is a strong solution to \eqref{ab eqn} belonging to the functional classes
$$
\begin{aligned}
&\mathrm{L}_{\infty,2, loc}\left( (0,T) \times \bR^d ,\frac{w(t)}{t}\mathrm{d}t \mathrm{d}x\right)\\
&\quad\cap \mathrm{L}_{2, loc}\left( (0,T) \times \bR^d,\frac{w(t)}{t^2}\mathrm{d}t\mathrm{d}x \right)\\
&\quad\cap \mathrm{H}_{1,2,loc}^{\psi}\left( (0,T) \times \bR^d\right),
\end{aligned}
$$
then $u$ is also a spatial Fourier weak solution to \eqref{ab eqn}.
\end{lem}
\begin{proof}
This assertion was established in Lemma \ref{strong unique lem}.
\end{proof}

We are now ready to prove our main theorems. Rather than following their initial order of presentation, we have sequenced the proofs to optimize logical flow. 
Consequently, we begin with the proof of Theorem \ref{main thm 2}, as it directly builds upon previous results and the groundwork already established in Section \ref{ex uni section}.

\vspace{2mm}
\begin{proof}[Proof of Theorem \ref{main thm 2}]

We first verify all hypotheses used below.  The first term in Assumption
\ref{main as 3} and Fubini's theorem show that, for almost every $\xi$,
$\psi(\cdot,\xi)\in\mathrm L_1(0,t)$ for every $t<T$; in particular,
Assumption \ref{main as 2} holds.  The same term gives the first part of
\eqref{20260811 01}.  Since $s^2/w(s)>0$ almost everywhere, the second term
in Assumption \ref{main as 3}, again by Fubini, gives
$\psi(s,\cdot)\in\mathrm L_{2,loc}(\bR^d)$ for almost every $s$.  Thus the
second part of \eqref{20260811 01} also holds.

To apply Corollary \ref{weak well weight}, it remains to verify condition
\eqref{20260324 01}. That is, we must show that
\begin{align*}
\int_{B_R} \int_0^t |\psi(\rho,\xi)| \int_0^\rho |\cF[f(s,\cdot)](\xi)| \mathrm{d}s \mathrm{d}\rho \mathrm{d}\xi < \infty \quad \forall R \in (0,\infty)~\text{and}~ \forall t \in (0,T).
\end{align*}
Applying Fubini's theorem, the Cauchy--Schwarz inequality, Hardy's inequality, and Plancherel's theorem in sequence, we obtain the following bounds:
\begin{align*}
&\int_{B_R} \int_0^t |\psi(\rho,\xi)| \int_0^\rho |\cF[f(s,\cdot)](\xi)| \mathrm{d}s \mathrm{d}\rho \mathrm{d}\xi \\
&= \int_0^t\int_0^\rho \int_{B_R} |\psi(\rho,\xi)|  |\cF[f(s,\cdot)](\xi)|\mathrm{d}\xi \mathrm{d}s \mathrm{d}\rho  \\
&\leq \int_0^t\int_0^\rho \left(\int_{B_R} |\psi(\rho,\xi)|^2 \mathrm{d}\xi\right)^{1/2} \left(\int_{\bR^d} |\cF[f(s,\cdot)](\xi)|^2 \mathrm{d}\xi \right)^{1/2} \mathrm{d}s \mathrm{d}\rho  \\
&\leq N_w \int_0^t\left(\int_{B_R} \frac{\rho^2}{w(\rho)} |\psi(\rho,\xi)|^2 \mathrm{d}\xi\right)^{1/2} \frac{1}{\rho} \int_0^\rho  \left(\int_{\bR^d} |\cF[f(s,\cdot)](\xi)|^2 \mathrm{d}\xi \right)^{1/2} \sqrt{w(s)} \mathrm{d}s \mathrm{d}\rho  \\
&\leq  N_w \left(\int_0^t \int_{B_R} \frac{s^2}{w(s)}|\psi(s,\xi)|^2 \mathrm{d}\xi \mathrm{d}s \right)^{1/2} \left(\int_0^t \int_{\bR^d} |\cF[f(s,\cdot)](\xi)|^2 \mathrm{d}\xi w(s) \mathrm{d}s \right)^{1/2}  \\
&\leq  N_w \left(\int_0^t \int_{B_R} \frac{s^2}{w(s)}|\psi(s,\xi)|^2 \mathrm{d}\xi \mathrm{d}s \right)^{1/2} \left(\int_0^t \int_{\bR^d} |f(s,x)|^2 \mathrm{d}x w(s) \mathrm{d}s \right)^{1/2} < \infty.
\end{align*}
Corollary \ref{weak well weight} now gives the solution \eqref{unique solution}
and the estimates \eqref{a priori est 2-2}--\eqref{a priori est 2-3}.

It remains to prove uniqueness in the class stated in the theorem.  Let
$u_1,u_2$ be two spatial Fourier weak solutions satisfying those estimates
and put $v=u_1-u_2$.  For every $R>0$ and $t<T$, Plancherel's theorem and
Cauchy--Schwarz give
\begin{align*}
&\int_0^t\int_{B_R}|\psi(s,\xi)|\,|\cF[v(s,\cdot)](\xi)|
\,\mathrm{d}\xi\mathrm{d}s\\
&\quad\leq
\left(\int_0^t\int_{B_R}\frac{s^2}{w(s)}|\psi(s,\xi)|^2
\,\mathrm{d}\xi\mathrm{d}s\right)^{1/2}
\left(\int_0^t\int_{\bR^d}\frac{w(s)}{s^2}
|\cF[v(s,\cdot)](\xi)|^2\,\mathrm{d}\xi\mathrm{d}s\right)^{1/2}<\infty.
\end{align*}
Subtracting the two weak formulations and using a countable dense family of
frequency test functions on each $B_R$ therefore yields
\[
\cF[v(t,\cdot)](\xi)
=\int_0^t\psi(s,\xi)\cF[v(s,\cdot)](\xi)\,\mathrm{d}s
\quad\text{for a.e. }(t,\xi).
\]
The already established time integrability of $\psi(\cdot,\xi)$ and the
integrating-factor argument from Theorem \ref{uniqueness theorem} imply
$v=0$.  Hence the solution is unique among spatial Fourier weak solutions
satisfying \eqref{a priori est 2-2}--\eqref{a priori est 2-3}.
\end{proof}

\vspace{3mm}
\begin{proof}[Proof of Theorem \ref{main thm}]

Since the a priori estimate \eqref{a priori est 2-2} was established in Corollary \ref{l2 bounded corollary 1}, our remaining task is simply to demonstrate that the function $u$ defined in \eqref{u defn} is the unique spectral-limit solution within the specified classes. 
We will address this in two stages: existence and uniqueness.

\vspace{2mm}
{\bf I. Existence.}
\vspace{2mm}

We first verify that Theorem \ref{main thm 2} applies to every truncated
symbol.  Since $\Re\psi\leq0$, the definition of $\psi_n$ gives
$\Re\psi_n\leq0$ and $|\psi_n|\leq\sqrt{2}\,n$.  Fix $t<T$ and choose
$\tau\in(t,T)$ as in the proof of Lemma \ref{strong unique lem}.  The
quasi-decreasing property gives $w(s)\geq w(\tau)/N_w$ for almost every
$s<t$.  Hence, for every $R>0$,
\begin{align*}
&\int_0^t\int_{B_R}\left(|\psi_n(s,\xi)|
+|\psi_n(s,\xi)|^2\frac{s^2}{w(s)}\right)\,\mathrm{d}\xi\mathrm{d}s\\
&\quad\leq |B_R|\left(\sqrt{2}\,nt
+\frac{2n^2N_w}{w(\tau)}\int_0^t s^2\,\mathrm{d}s\right)<\infty.
\end{align*}
Thus $\psi_n$ satisfies Assumptions \ref{main as} and \ref{main as 3}.
Theorem \ref{main thm 2}, applied to each bounded truncated symbol, yields a spatial Fourier weak solution to the corresponding truncated problem.
Specifically, for any test function $\varphi \in \cF^{-1}\mathrm{C}_c^\infty(\bR^d)$, 
there exists a spatial Fourier weak solution $u_n$  so that the following identity holds for almost every $t \in (0,T)$:
\begin{align*}
\left(u_n(t,\cdot),\varphi\right)_{\mathrm{L}_2(\bR^d)} =  \int_0^t \left(u_n(s,\cdot) , \overline \psi_n(s,-\mathrm{i}\nabla)\varphi \right)_{\mathrm{L}_2(\bR^d)} \mathrm{d}s 
+ \int_0^t \left( f(s,\cdot),\varphi\right)_{\mathrm{L}_2(\bR^d)} \mathrm{d}s
\quad a.e.~t\in (0,T),
\end{align*}
where the pseudo-differential operator is defined by
\begin{align*}
\psi_{n}(s,-\mathrm{i}\nabla)u_n(s,x)
=\cF^{-1} \left[  \psi_{n}(s,\cdot)  \cF[u_n(s,\cdot)]\right](x),
\end{align*}
and its corresponding truncated symbol is the symbol $\psi_n$ defined in Definition \ref{spectral-limit soluion}.
Theorem \ref{main thm 2} further guarantees that for each integer $n$, the solution takes the explicit form:
\begin{align*}
u_n(t,x)=  \cF^{-1}\left[ \int_0^t  \exp\left(\int_s^t\psi_n(r,\cdot)\mathrm{d}r \right) \cF[f(s,\cdot)]\mathrm{d}s \right](x).
\end{align*}
To establish existence, it is sufficient to prove that the sequence $\{u_n\}$ converges to the target function
\begin{align*}
u(t,x)=  \cF^{-1}\left[ \int_0^t  \exp\left(\int_s^t\psi(r,\cdot)\mathrm{d}r \right) \cF[f(s,\cdot)]\mathrm{d}s \right](x)
\end{align*}
in $\mathrm{L}_{2,loc}\left( (0,T) \times \bR^d\right)$.
Furthermore, under Assumption \ref{weight as}, we have the continuous embedding:
\begin{align}
										\label{20260324 50}
\mathrm{L}_{2,loc}\left( (0,T) \times \bR^d, w(t)\mathrm{d}t \mathrm{d}x \right) \hookrightarrow \mathrm{L}_{2,loc}\left( (0,T) \times \bR^d\right).
\end{align}
Consequently, our goal reduces to demonstrating that $u_n$ converges to $u$ in the weighted space 
$$
\mathrm{L}_{2,loc}\left( (0,T) \times \bR^d, w(t)\mathrm{d}t \mathrm{d}x \right).
$$
Fix $T'\in(0,T)$ and write
\[
E_n(t,s,\xi)=\exp\left(\int_s^t\psi_n(r,\xi)\,\mathrm{d}r\right),
\qquad
E(t,s,\xi)=\exp\left(\int_s^t\psi(r,\xi)\,\mathrm{d}r\right).
\]
For almost every $\xi$ and every $0<s<t<T'$, the definition of the truncation gives $\psi_n(r,\xi)\to\psi(r,\xi)$ and $|\psi_n(r,\xi)|\leq|\psi(r,\xi)|$. Assumption \ref{main as 2} therefore implies, by dominated convergence in $r$, that $E_n(t,s,\xi)\to E(t,s,\xi)$. Moreover, the non-positivity of the real parts gives $|E_n(t,s,\xi)|\leq1$ and $|E(t,s,\xi)|\leq1$.

Set
\[
G_n(t,\xi)=\int_0^t\bigl(E_n(t,s,\xi)-E(t,s,\xi)\bigr)\cF[f(s,\cdot)](\xi)\,\mathrm{d}s.
\]
For almost every $(t,\xi)$, dominated convergence in $s$ yields $G_n(t,\xi)\to0$. In addition, Jensen's inequality and the quasi-decreasing property of $w$ give
\begin{align*}
w(t)|G_n(t,\xi)|^2
&\leq 4t\,w(t)\int_0^t|\cF[f(s,\cdot)](\xi)|^2\,\mathrm{d}s\\
&\leq 4N_wT'\int_0^{T'}|\cF[f(s,\cdot)](\xi)|^2w(s)\,\mathrm{d}s.
\end{align*}
The last expression is independent of $t$ and is integrable on $(0,T')\times\bR^d$. Hence the dominated convergence theorem and Plancherel's identity yield
\begin{align*}
\int_0^{T'}\int_{\bR^d}|u_n(t,x)-u(t,x)|^2w(t)\,\mathrm{d}x\,\mathrm{d}t
=\int_0^{T'}\int_{\bR^d}|G_n(t,\xi)|^2w(t)\,\mathrm{d}\xi\,\mathrm{d}t
\longrightarrow0.
\end{align*}
Thus $u_n\to u$ in the required weighted local space, proving existence.

\vspace{2mm}
{\bf II. Uniqueness.}
\vspace{2mm}

By the embedding \eqref{20260324 50}, establishing uniqueness in the broader space $\mathrm{L}_{2,loc}\left( (0,T) \times \bR^d\right)$ is sufficient to guarantee uniqueness in our target space.
Let $u$ and $v$ be two spectral-limit solutions to \eqref{ab eqn}. By definition, there exist approximating sequences $\{u_n\}$ and $\{v_n\}$ in $\mathrm{L}_{2,loc}\left( (0,T) \times \bR^d \right)$ converging to $u$ and $v$, respectively. 
For each $n \in \bN$, both $u_n$ and $v_n$ serve as spatial Fourier weak solutions to \eqref{ab eqn} corresponding to the truncated symbol $\psi_n$ defined in Definition \ref{spectral-limit soluion}.
Since the truncated symbol $\psi_n$ is bounded, all hypotheses of Theorem \ref{main thm 2} are satisfied. 
Furthermore, Theorem \ref{uniqueness theorem} ensures that each truncated problem enjoys uniqueness within the class
\begin{align*}
\mathrm{L}_{2,loc}\left( (0,T) \times \bR^d \right) \cap \cF^{-1}\mathrm{L}_{1,2,loc}\left( (0,T) \times \bR^d, \vert{}\psi_n(t,\xi)\vert{}^2 \mathrm{d}t \mathrm{d}\xi \right) = \mathrm{L}_{2,loc}\left( (0,T) \times \bR^d \right).
\end{align*}
Consequently, Theorem \ref{main thm 2} yields an explicit, deterministic representation formula for $u_n$ and $v_n$, implying that $u_n = v_n$ for all $n$:
\begin{align*}
u_n(t,x) = \cF^{-1}\left[ \int_0^t \exp\left(\int_s^t \psi_n(r,\cdot)\mathrm{d}r \right) \cF[f(s,\cdot)]\mathrm{d}s \right] 
= v_n(t,x).
\end{align*}
Since $u_n = v_n$ pointwise for every $n$, their limits in $\mathrm{L}_{2,loc}\left( (0,T) \times \bR^d \right)$ must coincide. We conclude that $u(t,x) = v(t,x)$ almost everywhere on $(0,T) \times \bR^d$.

\end{proof}

\vspace{3mm}
\begin{proof}[Proof of Theorem \ref{main thm 3}]

\vspace{2mm}
{\bf I. Uniqueness and Setup}
\vspace{2mm}

As shown at the beginning of the proof of Theorem \ref{main thm 2},
Assumption \ref{main as 3} implies condition \eqref{20260811 01}.
The a priori bounds given in \eqref{a priori est 00} and \eqref{a priori est} have already been established by Lemma \ref{l2 bounded lemma} and Theorem \ref{continuity operator}.
 Furthermore, the combination of Theorem \ref{continuity operator} and Lemma \ref{strong unique lem} guarantees that any strong solution must be unique. Therefore, the remaining task is simply to demonstrate that a strong solution to \eqref{ab eqn} actually exists within the specified function class.

\vspace{2mm}
{\bf II. Existence via Approximation}
\vspace{2mm}

To prove existence, we use the first density assertion in Corollary \ref{dense cor} to choose a sequence $f_n\in\frac{1}{\sqrt w}\cA$
that converges to $f$ in
$\mathrm{L}_{2,loc}\left((0,T)\times\bR^d,
w(t)\mathrm{d}t\mathrm{d}x\right)$. Since Assumption
\ref{main as 3} holds, the calculation in the proof of Corollary
\ref{dense cor} shows that every $f_n$ belongs to $\mathrm{L}_{2,loc}\left((0,T)\times\bR^d,w(t)\mathrm{d}t\mathrm{d}x\right)\cap\cF^{-1}\mathrm{L}_{1,loc}\left((0,T)\times\bR^d\right)$ and satisfies \eqref{20260401 01}.
For every index $n$, Theorem \ref{strong solution thm} ensures we can construct a strong solution $u_n$, expressed as
$$
u_n(t,x) = \cF^{-1}\left[ \int_0^t \exp\left(\int_s^t\psi(r,\cdot)\mathrm{d}r \right) \cF[f_n(s,\cdot)]\mathrm{d}s \right](x).
$$
This sequence of solutions resides in the intersection space:
$$
\mathrm{L}_{\infty,2, loc}\left( (0,T) \times \bR^d ,\frac{w(t)}{t}\mathrm{d}t \mathrm{d}x\right) \cap \mathrm{L}_{2, loc}\left( (0,T) \times \bR^d,\frac{w(t)}{t^2}\mathrm{d}t\mathrm{d}x \right) \cap  \mathrm{H}_{1,2,loc}^{\psi}\left( (0,T) \times \bR^d\right).
$$
The additional $\phi$-regularity of $u_n$ follows from Lemma
\ref{l2 bounded lemma}, rather than from Theorem
\ref{continuity operator}. Indeed, Assumption \ref{main as 3} implies
Assumption \ref{main as 2}, while Assumption \ref{main as 5} includes
Assumption \ref{main as 4}. Hence, for every $T'\in(0,T)$,
\begin{align*}
\int_0^{T'}\left\|
\phi(-\mathrm{i}\nabla)u_n(t,\cdot)
\right\|_{\mathrm L_2(\bR^d)}^2w(t)\,\mathrm{d}t
\leq
N_w\int_0^{T'}\|f_n(t,\cdot)\|_{\mathrm L_2(\bR^d)}^2
w(t)\,\mathrm{d}t.
\end{align*}
Moreover, since $t^2\leq (T')^2$ on $(0,T')$, estimate
\eqref{20260321 04} gives
\begin{align*}
\int_0^{T'}\|u_n(t,\cdot)\|_{\mathrm L_2(\bR^d)}^2
w(t)\,\mathrm{d}t
&\leq
(T')^2\int_0^{T'}\|u_n(t,\cdot)\|_{\mathrm L_2(\bR^d)}^2
\frac{w(t)}{t^2}\,\mathrm{d}t\\
&\leq
4N_w(T')^2
\int_0^{T'}\|f_n(t,\cdot)\|_{\mathrm L_2(\bR^d)}^2
w(t)\,\mathrm{d}t.
\end{align*}
Consequently, $u_n\in
\mathrm{H}_{2,loc}^{\phi}
\left((0,T)\times\bR^d,w(t)\mathrm{d}t\mathrm{d}x\right)$.
Theorem \ref{continuity operator} may now be applied in its stated direction. 
It follows that $u_n\in\mathrm{H}_{1,2,loc}^{\psi}\left((0,T)\times\bR^d\right)$.

\vspace{2mm}
{\bf III. Convergence to the True Solution}
\vspace{2mm}

Fix an arbitrary $T'\in(0,T)$ and use the abbreviation
$$
\|g\|_{2,w;T'}^2
:=
\int_0^{T'}\|g(t,\cdot)\|_{\mathrm L_2(\bR^d)}^2
w(t)\,\mathrm{d}t.
$$
Also set
$$
A_{T'}
:=
\left(
\int_0^{T'}
\esssup_{\xi\in\bR^d}
\left|\frac{\psi(t,\xi)}{\phi(\xi)}\right|^2
\frac{1}{w(t)}\,\mathrm{d}t
\right)^{1/2},
$$
which is finite by Assumption \ref{main as 5}. For $n,m\in\bN$, put
$$
h_{n,m}:=f_n-f_m,
\qquad
z_{n,m}:=u_n-u_m.
$$
By linearity of the representation formula, $z_{n,m}$ is the function
given by \eqref{u defn} with $h_{n,m}$ in place of $f$. Therefore,
Corollary \ref{l2 bounded corollary 1} and Lemma
\ref{l2 bounded lemma} give
\begin{align*}
\int_0^{T'}\|z_{n,m}(t,\cdot)\|_{\mathrm L_2}^2
\frac{w(t)}{t^2}\,\mathrm{d}t
&\leq
4N_w\|h_{n,m}\|_{2,w;T'}^2,\\
\esssup_{0<t<T'}
\left(
\|z_{n,m}(t,\cdot)\|_{\mathrm L_2}^2
\frac{w(t)}{t}
\right)
&\leq
N_w\|h_{n,m}\|_{2,w;T'}^2,\\
\int_0^{T'}
\left\|\phi(-\mathrm{i}\nabla)z_{n,m}(t,\cdot)\right\|_{\mathrm L_2}^2
w(t)\,\mathrm{d}t
&\leq
N_w\|h_{n,m}\|_{2,w;T'}^2.
\end{align*}
In particular,
\begin{align*}
&\int_0^{T'}
\left(
\|z_{n,m}(t,\cdot)\|_{\mathrm L_2}^2
+
\left\|\phi(-\mathrm{i}\nabla)z_{n,m}(t,\cdot)\right\|_{\mathrm L_2}^2
\right)w(t)\,\mathrm{d}t\leq
N_w\bigl(1+4(T')^2\bigr)\|h_{n,m}\|_{2,w;T'}^2.
\end{align*}
Theorem \ref{continuity operator}, applied in the forward direction,
also yields
\begin{align*}
\int_0^{T'}
\left\|
\psi(t,-\mathrm{i}\nabla)z_{n,m}(t,\cdot)
\right\|_{\mathrm L_2}\,\mathrm{d}t
\leq
A_{T'}
\left(
\int_0^{T'}
\left\|\phi(-\mathrm{i}\nabla)z_{n,m}(t,\cdot)\right\|_{\mathrm L_2}^2
w(t)\,\mathrm{d}t
\right)^{1/2}\leq
A_{T'}\sqrt{N_w}\,
\|h_{n,m}\|_{2,w;T'}.
\end{align*}

Since $f_n\to f$ in the weighted local $\mathrm L_2$-space, these
estimates show that $\{u_n\}$ is Cauchy, on every $(0,T')$, in each of
the three spaces appearing in the statement of the theorem. The
fixed-interval $\phi$-potential space is complete by the remark
following Definition \ref{phi space}, and the other two fixed-interval
weighted spaces are complete as well.

We next verify that the local limits determine one function. Choose
$\tau\in(T',T)$ such that $0<w(\tau)<\infty$ and the quasi-decreasing
inequality holds for almost every $s<\tau$, and put
$$
c_{T'}:=\frac{w(\tau)}{N_w}>0.
$$
Then $w(s)\geq c_{T'}$ for almost every $s\in(0,T')$. Thus convergence
in any of the above fixed-interval spaces implies convergence locally
in the unweighted $\mathrm L_2$-space. Taking an increasing countable
sequence $T_k\uparrow T$, uniqueness of the unweighted limits on
overlapping intervals allows the fixed-interval limits to be patched
into a single function
$$
\begin{aligned}
u\in{}&
\mathrm{L}_{\infty,2,loc}
\left((0,T)\times\bR^d,\frac{w(t)}{t}\mathrm{d}t\mathrm{d}x\right)
\cap
\mathrm{L}_{2,loc}
\left((0,T)\times\bR^d,\frac{w(t)}{t^2}\mathrm{d}t\mathrm{d}x\right)\cap
\mathrm{H}_{2,loc}^{\phi}
\left((0,T)\times\bR^d,w(t)\mathrm{d}t\mathrm{d}x\right).
\end{aligned}
$$
By Theorem \ref{continuity operator},
$$
\psi(\cdot,-\mathrm{i}\nabla)u_n
\longrightarrow
\psi(\cdot,-\mathrm{i}\nabla)u
\quad\text{in }\mathrm L_1\left((0,T');\mathrm L_2(\bR^d)\right).
$$

We now pass to the limit in the strong equation. The lower bound
$w\geq c_{T'}$ gives
$$
\|f_n-f\|_{\mathrm L_1((0,T');\mathrm L_2)}
\leq
\left(\frac{T'}{c_{T'}}\right)^{1/2}
\|f_n-f\|_{2,w;T'}
\longrightarrow0,
$$
and convergence in the $\phi$-potential space similarly implies
$$
u_n\longrightarrow u
\quad\text{in }\mathrm L_1\left((0,T');\mathrm L_2(\bR^d)\right).
$$
Define
$$
G_n(t):=\psi(t,-\mathrm{i}\nabla)u_n(t)+f_n(t),
\qquad
G(t):=\psi(t,-\mathrm{i}\nabla)u(t)+f(t).
$$
The preceding convergences show that
$$
G_n\longrightarrow G
\quad\text{in }\mathrm L_1\left((0,T');\mathrm L_2(\bR^d)\right).
$$
Since $u_n$ is a strong solution,
$$
u_n(t)=\int_0^tG_n(s)\,\mathrm{d}s
\quad\text{in }\mathrm L_2(\bR^d)
\quad\text{for almost every }t\in(0,T').
$$
Moreover,
$$
\sup_{0\leq t\leq T'}
\left\|
\int_0^t\bigl(G_n(s)-G(s)\bigr)\,\mathrm{d}s
\right\|_{\mathrm L_2}
\leq
\|G_n-G\|_{\mathrm L_1((0,T');\mathrm L_2)}
\longrightarrow0.
$$
Combining this uniform convergence with
$u_n\to u$ in $\mathrm L_1((0,T');\mathrm L_2)$ gives
$$
u(t)
=
\int_0^t
\left(
\psi(s,-\mathrm{i}\nabla)u(s)+f(s)
\right)\,\mathrm{d}s
\quad\text{in }\mathrm L_2(\bR^d)
$$
for almost every $t\in(0,T')$. Since $T'<T$ was arbitrary, this is the
$\mathrm L_2$-valued Bochner form of \eqref{20230213 30}, and hence
$u$ is a strong solution in the sense of Definition
\ref{strong solution}.

Finally, let $\widetilde u$ denote the function defined by the right-hand
side of \eqref{unique solution}. Applying \eqref{20260321 04} to the
representation formula with $f_n-f$ in place of $f$ gives
$$
\int_0^{T'}
\|u_n(t,\cdot)-\widetilde u(t,\cdot)\|_{\mathrm L_2}^2
\frac{w(t)}{t^2}\,\mathrm{d}t
\leq
4N_w\|f_n-f\|_{2,w;T'}^2
\longrightarrow0.
$$
The uniqueness of the limit in this space therefore gives
$u=\widetilde u$ almost everywhere, proving the representation
\eqref{unique solution}. Passing to the limit in the preceding
estimates gives \eqref{a priori est 2-2},
\eqref{a priori est 2-3}, and \eqref{a priori est}; estimate
\eqref{a priori est 00} then follows from Theorem
\ref{continuity operator}.

\end{proof}

\bibliographystyle{plain}

\end{document}